\documentclass[10pt,a4paper,reqno,twoside]{amsart}
\usepackage{amsmath}
\usepackage{amssymb}
\usepackage{amsthm}
\usepackage{mathrsfs}
\usepackage{color}
\usepackage{amsmath}
\usepackage{amsfonts}
\usepackage{amsthm}
\newcommand{\xii}{{|\xi|}}

\newcommand{\R}{{\mathbb R}}

\newcommand{\Z}{{\mathbb Z}}
\newcommand{\N}{{\mathbb N}}

\newtheorem{theorem}{\bf Theorem}[section]
\newtheorem{lemma}{\bf Lemma}[section]
\newtheorem{proposition}{\bf Proposition}[section]
\newtheorem{corollary}{\bf Corollary}[section]

\theoremstyle{remark}
  \newtheorem{remark}{\sc Remark}[section]
\theoremstyle{definition}

 \numberwithin{equation}{section}

\title[]{Regularity theory and global existence of small data solutions to semi-linear de Sitter models with power non-linearity}
\author{ Marcelo Rempel Ebert \,\,\, and \,\,\,Michael Reissig}
\address{Marcelo Rempel Ebert, Departamento de Computa\c{c}\~{a}o e Matem\'atica, Universidade de S\~{a}o Paulo (USP), FFCLRP, Av. dos Bandeirantes, 3900, CEP 14040-901, Ribeir\~{a}o Preto - SP - Brasil}

\address{Michael Reissig, Faculty for Mathematics and Computer Science, Technical University Bergakademie Freiberg, Pr\"uferstr.9 - 09596 FREIBERG - GERMANY}

\begin{document}

\begin{abstract}
In this paper we study the Cauchy problem for semi-linear de Sitter models with power non-linearity. The model of interest is
\[ \phi_{tt} - e^{-2t} \Delta \phi + n\phi_t+m^2\phi=|\phi|^p,\quad (\phi(0,x),\phi_t(0,x))=(f(x),g(x)),\]
where $m^2$ is a non-negative constant. We study the global (in time) existence of small data solutions. In particular,
we show the interplay between the power $p$, admissible data spaces and admissible spaces of solutions (in weak sense, in sense of energy solutions or in classical sense).
\end{abstract}
\maketitle

\medskip


\section{Introduction} \label{Sec1}

In this paper we prove  global existence (in time) of small data
solutions of the Cauchy problem
\begin{equation}\label{general}
\begin{cases}
\phi_{tt} - e^{-2t} \Delta \phi + n\phi_t+m^2\phi=F(\phi),
  \quad (t,x)\in (0,\infty)\times \R^n,
\\
(\phi(0,x),\phi_t(0,x))=(f(x),g(x)),
  \quad x\in \R^n,
\end{cases}
\end{equation}
where $F(\phi)=|\phi|^p$, $p>1$ and $m>0$. This model describes the de Sitter model for the expansion of the universe.

If the  the initial condition
$\phi(0, x) = f(x)$ is small, then  $|\phi|^p$ becomes small for large $p$ and this term can be understood as a perturbation of the associate linear equation. For this
reason one is often able to prove such global (in time) existence result
only for some $p > p_0(m,n)>1$.  It is expected that  the dissipative effect in the considered model becomes more dominant with increasing parameters $m$ and $n$. Consequently, the function $p_0=p_0(m, n)$ could be expected to be a decreasing function in both variables $m$ and $n$.

In \cite{Y}, under the assumption that the right-hand side $F$ is  Lipschitz continuous in the Sobolev space $H^{s}(\R^n)$, $s> \frac{n}{2}$, the global existence (in time) of small data solutions to the model (\ref{general}) is proved for $m \in (0,\frac{\sqrt{n^2-1}}{2}) \cup [\frac{n}{2},\infty)$.
Some generalization of these results are obtained  in \cite{GY} including the case $m=\frac{\sqrt{n^2-1}}{2}$. In some cases, for instance, if the Cauchy problem has a  vanishing first initial data, the range $\big(\frac{\sqrt{n^2-1}}{2},\frac{n}{2}\big)$ for $m$ is allowed.
But in general, the case $m \in \big(\frac{\sqrt{n^2-1}}{2},\frac{n}{2}\big)$ remaind open.

On the other hand,  for $n\leq 4$ and $1+ \frac{4}{n} \leq p \leq \frac{n}{n-2}$,
the global existence  (in time)  of small data  energy solutions to \eqref{general} is proved for the case $m>\frac{n}{2}$ in \cite{N}.
There the range of admissible $p$ is bounded from below. So, a natural question appears: Is this restriction optimal or how does the admissible range of exponents
$p$ change with the choice of function spaces we take for the data and solutions?

Fortunately, the analysis of results of \cite{ERnew}  shed a light on the interval $\big(\frac{\sqrt{n^2-1}}{2},\frac{n}{2}\big)$ for $m$, too, and leads to the global existence of solutions for the Cauchy problem for the wave equation in the de Sitter model in different scales of Sobolev spaces. The main concerns of this paper are the following:
\begin{itemize}
\item To derive sharp estimates in scales of Sobolev spaces for the associated linear Cauchy problem \eqref{general} with right-hand side $0$.
\item  To prove global existence  of small data energy solutions in the supercritical case for all $m>0$. In particular, to derive results for $m \in \big(\frac{\sqrt{n^2-1}}{2},\frac{n}{2}\big)$.
\item  To show the interplay between the power $p$, admissible data spaces and admissible spaces for the solutions.
\end{itemize}

In order to derive our results, we apply the transformation
$\phi(t,x)=e^{r\,t}u(t,x)$
to the Cauchy
problem \eqref{general} to get (see \cite{Nthesis} or \cite{NunesPalmieriReissig} in the case of constant speed of propagation)
\begin{equation}\label{semilieardampedwave}
\begin{cases}
u_{tt} - e^{-2t}\Delta u + (2r+n) u_t + (r^2+rn+m^2)u=e^{(p-1)rt}|u|^p,
  \quad (t,x)\in (0,\infty)\times \R^n,
\\
u(0,x)= f(x)=: u_0(x), \quad u_t(0,x)=g(x)-rf(x)=: u_1(x),
  \quad x\in \R^n.
\end{cases}
\end{equation}
Then we split our analysis into three cases:
\subsection{De Sitter model with dominant dissipation} \label{Sec2.1}
If we choose in (\ref{semilieardampedwave}) the parameter
\[r:=\frac{-n + \mu}{2}, \qquad \mu:=\sqrt{n^2-4m^2},\]
then  $r^2+rn+m^2=0$ and we obtain the model with dominant
dissipation
\begin{equation}\label{semilieardampedwavedissipation}
\begin{cases}
u_{tt} - e^{-2t}\Delta u + \mu\, u_t =e^{r(p-1)t}|u|^p,
  \quad (t,x)\in (0,\infty)\times \R^n,
\\
u(0,x)= f(x)=: u_0(x), \quad u_t(0,x)=g(x)+ \frac{n - \sqrt{n^2-4m^2}}{2}f(x)=: u_1(x),
  \quad x\in \R^n.
\end{cases}
\end{equation}
In the following we want to have some improving influence of the dissipation term. For this reason we assume $m < \frac{n}{2}$.
In the paper \cite{ERnew} the authors introduced some classification of damping terms for the Cauchy problem
\[v_{tt} - a(t)^2\Delta v + b(t)v_t =0, \qquad
v(0,x)=v_0(x), \qquad v_t(0,x)=v_1(x).\] Due to this classification it turns out that the dissipative term in (\ref{semilieardampedwavedissipation}) is non-effective if $\mu \in (0,1)$, that is, $m \in (\frac{\sqrt{n^2-1}}{2},\frac{n}{2})$. In the case of non-effective damping term the treatment of semi-linear wave models with power non-linearity $|u|^p$ is an open problem up to now. Some special cases are treated in \cite{DabbiccoLucenteReissig}. To avoid at the beginning the non-effectiveness we assume $\mu \geq 1$, that is, $m \in (0,\sqrt{n^2-1}/2]$. This case is treated in Section \ref{Sec3.2}. But, finally,  we will present in Section \ref{Sec3.3} results in the case
of non-effective dissipation, too.
In a first step one should derive, similar to \cite{Bui1} or \cite{Bui2}, estimates for solutions to linear
damped wave equations with time-dependent speed of propagation, but now for solutions to Cauchy problems with parameter-dependent Cauchy conditions
\[v_{tt} - a(t)^2\triangle v + b(t)v_t =0, \qquad
v(s,x)=v_0(s, x), \qquad v_t(s,x)=v_1(s, x). \]

\subsection{De Sitter model with dominant mass} \label{Sec2.2}
If we choose in (\ref{semilieardampedwave}) the parameter $r=-\frac{n}{2}$, then we obtain the model with dominant
mass
\begin{equation}\label{semilieardampedwavemass}
\begin{cases}
u_{tt} - e^{-2t}\Delta u + \big(m^2-\frac{n^2}{4}\big) \,u=e^{-\frac{n}{2}(p-1)t}|u|^p,
  \quad (t,x)\in (0,\infty)\times \R^n,
\\
u(0,x)= f(x)=: u_0(x), \quad u_t(0,x)=g(x)+\frac{n}{2}f(x)=: u_1(x),
  \quad x\in \R^n.
\end{cases}
\end{equation}
In the following we want to have some improving influence of the mass term. For this reason we assume $m > \frac{n}{2}$.
In a first step one should derive, similar to \cite{B} or \cite{BöhmeReissig}, estimates for solutions to linear
Klein-Gordon equations with time-dependent speed of propagation, but now for solutions to Cauchy problems with parameter-dependent Cauchy conditions
\[v_{tt} - a(t)^2\triangle v + m(t)^2v =0, \qquad
v(s,x)=v_0(s, x), \qquad v_t(s,x)=v_1(s, x). \]
This model is treated in Section \ref{Sec4}.
\subsection{De Sitter model with balanced dissipation and mass} \label{Sec2.3}
If we choose in (\ref{semilieardampedwave}) the parameters $r=-\frac{n}{2}$ and $m=\frac{n}{2}$, then we obtain the model with a balance between mass and dissipation
\begin{equation}\label{semilieardampedwavebalanced}
\begin{cases}
u_{tt} - e^{-2t}\Delta u =e^{-\frac{n}{2}(p-1)t}|u|^p,
  \quad (t,x)\in (0,\infty)\times \R^n,
\\
u(0,x)= f(x)=: u_0(x), \quad u_t(0,x)=g(x)+\frac{n}{2}f(x)=: u_1(x),
  \quad x\in \R^n.
\end{cases}
\end{equation}
This model is treated in Section \ref{Sec5}.\\
The present paper is organized as follows. In Sections \ref{Sec3} to \ref{Sec5} we consider the cases  $m \in (0,\frac{n}{2})$, $m >\frac{n}{2}$ and $m=\frac{n}{2}$, respectively.
    An appendix containing some tools of Harmonic Analysis completes the paper.

\section{De Sitter model with dominant dissipation: Case $m\in (0,\frac{n}{2})$.} \label{Sec3}
In this section we consider the Cauchy problem
\begin{equation}\label{semilieardampedwavedissipationtreatment}
\begin{cases}
u_{tt} - e^{-2t}\Delta u + \mu \, u_t =e^{r (p-1)t}|u|^p,
  \quad (t,x)\in (0,\infty)\times \R^n,
\\
u(0,x)= u_0(x), \quad u_t(0,x)=u_1(x),
  \quad x\in \R^n
\end{cases}
\end{equation}
with $\mu=\sqrt{n^2 - 4m^2} \in (0,n)$ and $r:=\frac{-n + \mu}{2}$.\\
According to Duhamel's principle, a solution of \eqref{semilieardampedwavedissipationtreatment} satisfies the non-linear integral equation
\[  u(t,x)=K_0(t,0, x) \ast_{(x)} u_0(x) + K_1(t,0, x) \ast_{(x)} u_1(x) + \int_0^t e^{(p-1)rs}K_1(t,s,x) \ast_{(x)}|u(s,x)|^{p}\,ds, \]
where $K_j(t,0, x) \ast_{(x)} u_j(x)$, $j=0,1$, are the solutions to the corresponding linear Cauchy
problem
\begin{equation}\label{dampedwave}
\begin{cases}
u_{tt} - e^{-2t}\Delta u + \mu u_t=0,
  \quad (t,x)\in (0,\infty)\times \R^n,
\\
u(0,x)=\delta_{0j}u_0(x), \quad u_t(0,x)=\delta_{1j}u_1(x),
  \quad x\in \R^n,
\end{cases}
\end{equation}
with $\delta_{kj}=1$  for $k=j$, and zero otherwise. The term $K_1(t,s,x) \ast_{(x)}f(s,x)$ is the  solution  of the parameter-dependent Cauchy problem
\begin{eqnarray*}
&& u_{tt} - e^{-2t}\Delta u + \mu u_t=0,
  \quad (t,x)\in (s,\infty)\times \R^n,
\\
&& u(s,x)=0, \quad u_t(s,x)=f(s,x),
  \quad x\in \R^n.
\end{eqnarray*}
So, Duhamel's principle explains that we have to take account of solutions
to a family of parameter-dependent Cauchy problems.
\subsection{Estimates of solutions to the corresponding linear model} \label{Sec3.1}
Let us consider for $\mu >0$ the parameter-dependent Cauchy problem for the damped wave equation
\begin{equation}\label{dampedwavelinear}
\begin{cases}
u_{tt} - e^{-2t}\Delta u + \mu u_t=0,
  \quad (t,x)\in (s,\infty)\times \R^n,
\\
u(s,x)= \varphi(s,x), \quad u_t(s,x)=\psi(s,x),
  \quad x\in \R^n.
\end{cases}
\end{equation}
We perform the partial Fourier transformation with respect to the spatial variables to \eqref{dampedwavelinear} and the change of variables
\[ v(\tau)=\widehat{u}(t,\xi), \quad \tau=A(t)|\xi|, \quad A(t)=\int_t^{\infty}e^{-\tau}d\tau=e^{-t}.\]
All this leads to the following Cauchy problem:
\[v_{\tau\tau} +\frac{1-\mu}{\tau} v_{\tau}+ v=0,
  \quad
v(\xii e^{-s})=\widehat\varphi(s,\xi), \quad v_{\tau}(\xii e^{-s})=-\frac{\widehat\psi(s,\xi)}{|\xi|e^{-s}}.\]
Now, setting
\[w(z)=e^{\frac{z}{2}}v(\tau), \quad z=2i\tau=2i\xii e^{-t},\]
we get the confluent hypergeometric equation
\begin{equation}\label{hyperdamping}
zw_{zz}+ \big(1-\mu-z\big)w_z-\frac{1-\mu}{2}w=0,
\end{equation}
and the following initial conditions at $z_0=z_0(s,\xi)=2i\xii e^{-s}$:
\[w(z_0)=e^{i\xii e^{-s}}\widehat\varphi(s,\xi), \quad w_z(z_0)=\frac{e^{i\xii e^{-s}}}{2}\Big(\widehat\varphi(s,\xi)+i\frac{\widehat\psi(s,\xi)}{|\xi|e^{-s}}\Big).\]
If $\mu\notin \Z$, then due to \cite{BE} the general solution of \eqref{hyperdamping} has the representation
\[w(z)=c_1(s,\xi)w_1(z)+ c_2(s,\xi)w_2(z),\]
where $w_1$ and $w_2$ are
 two linear independent solutions  given by
 \[w_1(z)=\Phi\Big(\frac{1-\mu}{2}, 1-\mu, z\Big), \quad w_2(z)=z^{\mu}\Phi\Big(\frac{1+\mu}{2}, 1+\mu, z\Big). \]
Here $\Phi$ is the Kummer's function
\begin{equation}\label{Kummer}
 \Phi(b,c,x)= \sum_{n=0}^{\infty} \frac{ b^{(n)} z^n}{c^{(n)}n!},\,\,\,a^{(n)}=\prod_{k=0}^{n-1} (a+k),\,\,\,a^{(0)}=1.
\end{equation}
The function $\Phi$ is  an entire function of $b$ and $z$,  except when $c = 0, -1, -2, \cdots$. As a function of $c$ it is analytic except for poles at the non-positive integers.
 Moreover, we can write
 \begin{equation}\label{coef}
 c_j(s,\xi)=(-1)^{3-j}\frac{w(z_0)(d_z w_{3-j})(z_0)-(d_z w)(z_0)w_{3-j}(z_0)}{W(w_1,w_2)(z_0)} \,\,\,\mbox{for} \,\,\, j=1,2,
 \end{equation}
 where $W(w_1,w_2)$ is the Wronskian of the two linear independent solutions and it satisfies (\cite{BE},vol.1,p.253)
\[W(w_1,w_2)(z)=w_1\frac{d}{dz}w_2-w_2\frac{d}{dz}w_1=\mu z^{\mu-1}e^z.\]
Using all these functions we conclude the following WKB representation
for $\widehat{u}$:
\begin{eqnarray*}
  \widehat{u}(t,\xi)=e^{-i\xii e^{-t}}\Big(c_1(s,\xi)\Phi\Big(\frac{1-\mu}{2}, 1-\mu, 2i\xii e^{-t}\Big)+ c_2(s,\xi) (2i\xii e^{-t})^{\mu}\Phi\Big(\frac{1+\mu}{2}, 1+\mu, 2i\xii e^{-t} \Big) \Big),
   \end{eqnarray*}
  where $c_j(s,\xi)$ for $j=1,2$ are given by \eqref{coef} with $z_0=2i\xii e^{-s}$.\\
  So, to describe the asymptotic behavior of $\widehat{u}$, we may use the following well-known properties of
the function $\Phi$ (see \cite{BE}):
\begin{proposition}\label{hypergeometricFucntions}
Let $b$ and $c$ be fixed parameters in $\mathbb{C}$ with $c \notin \Z$.
Then the function $\Phi$ satisfies the following properties:
\begin{itemize}
\item (P1): $\Phi=\Phi(b,c,z)$ is an entire function with respect to $z$;
\item (P2): $\frac{d}{dz}\Phi(b,c,z)=\frac{b}{c}\Phi(b+1,c+1,z)$;
\item (P3): the behavior for large $|z|$ is given by
\[ |\Phi(b,c,z)|\leq C_{b,c} |z|^{\max\{\Re(b-c), -\Re b\}}.\]
\end{itemize}
\end{proposition}
In order to have pointwise estimates for $\widehat{u}$ and their derivatives, we analyze the behavior of $\Phi(b,c,\cdot)$ for small and large arguments. For this reason we split the extended phase space into three zones
\[ Z_1(s)
     := \{\xi: \ \xii e^{-s}\leq N \},\ \ Z_2(t,s)
     := \{N e^{s}\leq \xii \leq N e^{t} \},  \ \ Z_3(t)
     := \{\xii e^{-t}\geq N \}.\]
     \begin{enumerate}
     \item
     In $Z_1$, by using properties (P1) and (P2) we have
\[ |c_1(s, \xi)|\lesssim (\xii e^{-s})^{1-\mu}\Big(|\widehat\varphi(s,\xi)|(\xii e^{-s})^{\mu-1}+\Big(|\widehat\varphi(s,\xi)|+\frac{|\widehat\psi(s,\xi)|}{\xii e^{-s}}\Big)(\xii e^{-s})^{\mu}\Big)
\lesssim|\widehat\varphi(s,\xi)|+|\widehat\psi(s,\xi)|,  \]
and
\[ |c_2(s,\xi)|\lesssim (\xii e^{-s})^{1-\mu}\Big(|\widehat\varphi(s,\xi)|+\frac{|\widehat\psi(s,\xi)|}{\xii e^{-s}}\Big).\]
So we can estimate for $t \geq s$
\[|\widehat{u}(t,\xi)| \lesssim |\widehat\varphi(s,\xi)|+|\widehat\psi(s,\xi)|+ (\xii e^{-s})^{1-\mu}\Big(|\widehat\varphi(s,\xi)|+\frac{|\widehat\psi(s,\xi)|}{\xii e^{-s}}\Big)(\xii e^{-t})^{\mu}\lesssim |\widehat\varphi(s,\xi)|+|\widehat\psi(s,\xi)|,\]
and, more general,
\[\xii^\gamma |\widehat{u}(t,\xi)|\lesssim \xii^\gamma |\widehat\varphi(s,\xi)|+\xii^{\gamma_1}e^{\gamma_2 s}|\widehat\psi(s,\xi)| \,\,\,\mbox{for}\,\,\, \gamma, \gamma_1,\gamma_2 \geq 0\,\,\,\mbox{and} \,\,\,\gamma=\gamma_1+\gamma_2.\]
 Similarly, if $\mu<1$, then we have for $\gamma \geq 1$ the estimates
 \[\xii^{\gamma-1} |\widehat{u}_t(t,\xi)|\lesssim e^{\mu(s-t)} \big(e^{-s}\xii^\gamma|\widehat\varphi(s,\xi)|+\xii^{\gamma-1} |\widehat\psi(s,\xi)|\big),\]
whereas for $\mu\geq 1$ we may conclude
\[ \xii^{\gamma-1}|\widehat{u}_t(t,\xi)|\lesssim e^{-t}\xii^\gamma|\widehat\varphi(s,\xi)|+e^{(s-t)}\xii^{\gamma-1}|\widehat\psi(s,\xi)|.\]
If we additionally assume $\varphi, \psi \in H^\gamma$, $\gamma >0$, then we may avoid any loss of decay and may derive
\[\xii^\gamma |\widehat{u}(t,\xi)|\lesssim \xii^\gamma \big(|\widehat\varphi(s,\xi)|+|\widehat\psi(s,\xi)|\big) \,\,\,\mbox{for}\,\,\, \gamma >0.\]
\item In $Z_3$ we have $\xii e^{-t}\geq N$ and $\xii e^{-s}\geq N$. Thanks to properties (P2) and (P3), we can estimate for $j=1,2$
\begin{equation}\label{coefdamp}
  \begin{cases} |c_j(s, \xi)|\lesssim\!(\xii e^{-s})^{1-\mu}\!\Big(\!|\widehat\varphi(s,\xi)|+\frac{|\widehat\psi(s,\xi)|}{\xii e^{-s}}\!\Big)
  \,(\xii e^{-s})^{\frac{\mu-1}{2}}\\ \qquad \lesssim\!|\widehat\varphi(s,\xi)|(\xii e^{-s})^{\frac{1-\mu}{2}}+|\widehat\psi(s,\xi)|(\xii e^{-s})^{-\frac{1+\mu}{2}}.
  \end{cases}
 \end{equation}
  So,  by using \eqref{coefdamp} and again property (P3) we conclude
  \begin{eqnarray*}
  &\xii^\gamma |\widehat{u}(t,\xi)| \lesssim  \xii^\gamma \left(|\widehat\varphi(s,\xi)|(\xii e^{-s})^{\frac{1-\mu}{2}}+|\widehat\psi(s,\xi)|(\xii e^{-s})^{-\frac{1+\mu}{2}}\right)
  (\xii e^{-t})^{\frac{\mu-1}{2}}\\
  &\lesssim
  \xii^\gamma \left(|\widehat\varphi(s,\xi)|e^{\frac{(\mu-1)s}{2}}+ \frac{|\widehat\psi(s,\xi)|}{\xii}e^{\frac{(1+\mu)s}{2}}\right)
   e^{\frac{(1-\mu)t}{2}}\\
  &\lesssim
   \xii^\gamma \left(|\widehat\varphi(s,\xi)|e^{\frac{(\mu-1)(s-t)}{2}}+ |\widehat\psi(s,\xi)|e^{\frac{(1+\mu)(s-t)}{2}}\right)\,\,\,\mbox{for}\,\,\, \gamma \geq 0.
  \end{eqnarray*}
  In order to avoid any exponential increasing term in $s$, we use regularity in the last inequality, i.e., we use the estimate $N\xii^{-1}\leq e^{-t}$. If $\psi\in H^{\gamma-1}$, but does not belong to $H^\gamma$, one may only derive
 \[\xii^\gamma|\widehat{u}(t,\xi)|\lesssim \xii^\gamma |\widehat\varphi(s,\xi)|e^{\frac{(\mu-1)(s-t)}{2}}+ \xii^{\gamma-1}|\widehat\psi(s,\xi)|e^{\frac{(1+\mu)s}{2}}e^{\frac{(1-\mu)t}{2}}.\]
  Similarly, we conclude for $\gamma \geq 1$ the estimate
 \[\xii^{\gamma-1}|\widehat{u}_t(t,\xi)|\lesssim e^{-t}e^{\frac{\mu-1}{2}(s-t)}\xii^\gamma|\widehat\varphi(s,\xi)|+e^{\frac{1+\mu}{2}(s-t)}\xii^{\gamma-1}|\widehat\psi(s,\xi)|. \]
\item In $Z_2$ we still use \eqref{coefdamp} to conclude
  \begin{eqnarray*}
  & \xii^\gamma|\widehat{u}(t,\xi)|  \lesssim  \xii^\sigma \left(|\widehat\varphi(s,\xi)|(\xii e^{-s})^{\frac{1-\mu}{2}}+|\widehat\psi(s,\xi)|(\xii e^{-s})^{-\frac{1+\mu}{2}}\right)
  \left(1+(\xii e^{-t})^{\mu}\right)\\
 & \lesssim  \xii^\gamma \left(|\widehat\varphi(s,\xi)|(\xii e^{-t} e^t e^{-s})^{\frac{1-\mu}{2}}+|\widehat\psi(s,\xi)|(\xii e^{-t} e^t e^{-s})^{-\frac{1+\mu}{2}}\right)
  \left(1+(\xii e^{-t})^{\mu}\right)\\&\lesssim \xii^\gamma \left(|\widehat\varphi(s,\xi)|e^{\frac{(\mu-1)(s-t)}{2}} + |\widehat\psi(s,\xi)|e^{\frac{(\mu+1)(s-t)}{2}}\right) \,\,\,\mbox{for}\,\,\, \gamma \geq 0.
  \end{eqnarray*}
  Again, if $\psi\in H^{\gamma-1}$, but does not belong to $H^\gamma$, one may only derive the estimate
 \begin{eqnarray*}
  &\xii|^\gamma \widehat{u}(t,\xi)|\lesssim \xii^\gamma |\widehat\varphi(s,\xi)|e^{\frac{(\mu-1)(s-t)}{2}}+ \xii^{\gamma-1} |\widehat\psi(s,\xi)|e^{\frac{(1+\mu)s}{2}}e^{\frac{(1-\mu)t}{2}}\\
  &\lesssim \xii^\gamma |\widehat\varphi(s,\xi)|e^{\frac{(\mu-1)(s-t)}{2}}+ \xii^{\gamma-1} |\widehat\psi(s,\xi)|e^{\frac{(\mu-1)(s-t)}{2}}e^{s}.
  \end{eqnarray*}
  Similarly, we conclude
\[\xii^{\gamma-1} |\widehat{u}_t(t,\xi)|\lesssim e^{-t}e^{\frac{\mu-1}{2}(s-t)}\xii^\gamma |\widehat\varphi(s,\xi)|+e^{\frac{1+\mu}{2}(s-t)}\xii^{\gamma-1} |\widehat\psi(s,\xi)|. \]
\end{enumerate}
Now, let us devote to the case $\mu \in \N$. For $\mu=1$ the proof  follows immediately by using the explicit representation for the solution
 to the Cauchy problem \eqref{dampedwavelinear}, that is, the relation
 \[ \widehat{u}(t,\xi)=\cos\big(|\xi|(e^{-s}-e^{-t})\big)\widehat\varphi(s,\xi) + \sin \big(|\xi|(e^{-s}-e^{-t})\big)\frac{\widehat\psi(s,\xi)}{|\xi|e^{-s}}. \]
 If $\mu \geq 2$ and $\mu \in \mathbb{N}$, then the function $\Phi(b,c,x)$ given by \eqref{Kummer} with $c=1-\mu$ and $b=\frac{c}{2}$ is no longer well-defined.
 In these cases, $w_2(z)=z^{1-c}\Phi\big(b-c+1, 2-c, z\big)$ is still one solution and by using Frobenius' method or Laplace transform one may  find a second linear independent solution $\Psi (b, c, z)$  to Kummer's equation
 \begin{equation*}
zw_{zz}+ \big(c-z\big)w_z-bw=0
\end{equation*}
satisfying the following properties(see \cite{BE}, pages 256, 260, 262 and 278):
\begin{itemize}
\item $W(w_2(z), \Psi (b, c, z) )=\frac{\Gamma(2-c)}{\Gamma(b-c+1)} z^{-c}e^z$;
\item  $\frac{d}{dz}\Psi(b,c,z)=\Psi(b, c,z)+ \Psi(b,c+1,z)$;
\item for $\Re c<0$ the behavior for small $|z|$ is given by
\[ \Psi(b,c,z)= \frac{\Gamma(1-c)}{\Gamma(b-c+1)} + O(|z|),\]
where $\Gamma$ denotes  the Gamma function;
\item for large $|z|$ the behavior is given by
\[ |\Psi(b,c,z)|\leq C_{b,c} |z|^{ - b}.\]
\end{itemize}
Thanks to $b-c=-b=\frac{\mu-1}{2}$ we may conclude that   $\Psi(b,c,z)$ and $\frac{d}{dz}\Psi(b,c,z)$ satisfy  the same estimates  as  $\Phi(b,c,z)$ and $\frac{d}{dz}\Phi(b,c,z)$ in the case $c \notin \Z$. \\

Summing up, we have proved the following result:
\begin{proposition}\label{proposition1}
Assume that $\varphi(0,\cdot) \in \dot{H}^\gamma(\R^n)$ with $\gamma \geq 0$ and $\psi(0,\cdot) \equiv 0$. Then the following estimates hold for $t \in [0,\infty)$:\\
If $\mu \in (0,1)$, then
 \begin{equation} \label{optidamp}
  \|K_0(t,0,x) \ast_{(x)} \varphi(0,x)\|_{\dot{H^\gamma}} \lesssim
  e^{\frac{(1-\mu)t}{2}}\|\varphi(0,x)\|_{\dot{H^\gamma}},
  \end{equation}
 and
    \begin{equation} \label{optidamptimederivative}
  \|\partial_t K_0(t,s, x) \ast_{(x)} \varphi(s,x)\|_{\dot{H}^{\gamma-1}} \lesssim e^{-\mu t}\|\varphi(0,x)\|_{\dot{H}^{\gamma}}\,\,\,\mbox{for}\,\,\,\gamma \geq 1.
   \end{equation}
If $\mu\in [1,n)$, then
 \begin{equation} \label{optidamp1}
  \| K_0(t,0, x) \ast_{(x)} \varphi(0,x)\|_{\dot{H^\gamma}} \lesssim
 \|\varphi(0,x)\|_{\dot{H^\gamma}},
  \end{equation}
  and
\begin{equation} \label{optidamp1timederivative}
    \|\partial_t K_0(t,s, x) \ast_{(x)} \varphi(s,x)\|_{\dot{H}^{\gamma-1}} \lesssim e^{-t}\|\varphi(0,x)\|_{\dot{H}^{\gamma}}\,\,\,\mbox{for}\,\,\,\gamma \geq 1.
   \end{equation}
   \end{proposition}
 \begin{proposition}\label{proposition11}
 Assume that $\psi(s,\cdot)\in \dot{H}^\gamma(\R^n)$ with $\gamma \geq 0$ for $s \in [0,\infty)$ and $\varphi(s,\cdot) \equiv 0$. Then the following estimates hold:\\
 We have for all  $\mu \in (0,n)$ and $t \in [s,\infty)$ the estimates
  \begin{equation} \label{optidamp2}
  \|  K_1(t,s, x) \ast_{(x)} \psi(s,x)\|_{\dot{H}^\gamma} \lesssim
 \|\psi(s,x)\|_{\dot{H}^\gamma}.
  \end{equation}
 If $\mu \in (0,1)$, then
  \begin{equation} \label{optidampwithloss}
  \| \nabla K_1(t,s, x) \ast_{(x)} \psi(s,x)\|_{\dot{H}^\gamma} \lesssim
 e^{\frac{(1+\mu)s}{2}}e^{\frac{(1-\mu)t}{2}}\|\psi(s,x)\|_{\dot{H}^\gamma},
  \end{equation}
  and
    \begin{equation} \label{optidamp3timederivative}
  \|\partial_t K_1(t,s, x) \ast_{(x)} \psi(s,x)\|_{\dot{H}^\gamma}\lesssim e^{\mu(s- t)}\|\psi(s,x)\|_{\dot{H}^\gamma}.
   \end{equation}
  For all  $ \mu\in [1,n)$ and $t \in [s,\infty)$ we have
  \begin{equation} \label{optidamp3}
  \| \nabla K_1(t,s, x) \ast_{(x)} \psi(s,x)\|_{\dot{H}^\gamma} \lesssim
 e^{s}\|\psi(s,x)\|_{\dot{H}^\gamma},
  \end{equation}
   and
    \begin{equation} \label{optidamp4timederivative}
  \|\partial_t K_1(t,s, x) \ast_{(x)} \psi(s,x)\|_{\dot{H}^\gamma}\lesssim e^{(s- t)}\|\psi(s,x)\|_{\dot{H}^\gamma}.
   \end{equation}
\end{proposition}
 \begin{remark} If we are interested to estimate the norm  \[\| \nabla K_1(t,s, x) \ast_{(x)} \psi(s,x)\|_{\dot{H}^{\gamma-1}},\] then the estimates (\ref{optidamp2}) and (\ref{optidamp3}) show that we have a benefit by assuming $\psi(s,\cdot) \in \dot{H}^{\gamma}$ instead of $\psi(s,\cdot) \in \dot{H}^{\gamma-1}$ only. On the contrary,
 one can not expect any benefit in the estimates for the norm
   \[\|\partial_t K_1(t,s, x) \ast_{(x)} \psi(s,x)\|_{\dot{H}^{\gamma-1}}\] by using  additional $\dot{H}^{\gamma}$ regularity.
    \end{remark}
  \begin{corollary}\label{lineardecayestimates}
  Consider the Cauchy problem \eqref{general} with a vanishing right-hand side. Assume that $f \in H^\gamma(\R^n)$ and $g \in H^{\gamma-1}(\R^n)$ with $\gamma \geq 1$. Then the solution $\phi$ satisfies the following a-priori estimates with the parameter $\mu=\sqrt{n^2-4m^2}\in (0,n)$: \\
If $\mu \in (0,1)$, then
    \begin{equation} \label{optidamp3energy}
 \|\phi(t,\cdot)\|_{H^\gamma} \lesssim
  e^{-\frac{(n-1)t}{2}}\big(\|f\|_{H^\gamma}+ \|g\|_{H^{\gamma-1}} \big),
  \end{equation}
  and
  \begin{equation} \label{optidamp4energy}
  \| \phi_t(t,\cdot)\|_{H^{\gamma-1}} \lesssim
  e^{-\frac{(n-1)t}{2}}\big(\|f\|_{H^\gamma}+ \|g\|_{H^{\gamma-1}} \big),
  \end{equation}
whereas for $\mu \in [1,n)$ we conclude
      \begin{equation} \label{optidamp4}
  \| \phi(t,\cdot)\|_{H^\gamma} \lesssim
  e^{\frac{(-n+\mu)t}{2}}\big(\|f\|_{H^\gamma}+ \|g\|_{H^{\gamma-1}} \big),
 \end{equation}
 and
 \begin{equation} \label{optidamp5}
  \|\phi_t(t,\cdot)\|_{H^{\gamma-1}} \lesssim
  e^{\frac{(-n+\mu)t}{2}}\big(\|f\|_{H^\gamma}+ \|g\|_{H^{\gamma-1}} \big).
 \end{equation}
 If we additionally assume $g \in H^\gamma(\R^n)$, then  the estimate \eqref{optidamp3energy} improves for $\mu\in (0,1)$ to
   \begin{equation} \label{optidamp31}
  \| \phi(t,\cdot)\|_{H^{\gamma}} \lesssim
  e^{-\frac{(n-1)t}{2}}\big(\|f\|_{H^\gamma}+ e^{-\frac{(1-\mu) t}{2}}\|g\|_{H^\gamma} \big).
  \end{equation}
\end{corollary}
\begin{remark} If we take $g\equiv 0$ in Corollary \ref{lineardecayestimates}, then we have a better decay estimate  for $\mu\in (0,1)$ than for $\mu\in (1,n)$.
  The reason is that $\mu(m)$ is a decreasing function in $m$ for $0<m<\frac{n}{2}$.
  \end{remark}
  \subsection{Global existence of small data solutions: Case $m\in (0, \frac{\sqrt{n^2-1}}{2}]$} \label{Sec3.2}

Firstly we are interested in the global existence (in time) of energy solutions.
\begin{theorem}\label{main2} Consider for $n \geq 2$ the Cauchy problem \eqref{general} with data $f \in H^1(\R^n)$ and $g \in L^2(\R^n)$. Let $p>p_{n,\mu}:=1+  \frac{2}{n-\mu}$ and $p\leq \frac{n}{n-2}$ for $n\geq 3$. Assume that the parameter $\mu=\sqrt{n^2-4m^2}$ satisfies $\mu \in [1,2)$, i.e.,  $m\in \big(\frac{\sqrt{n^2-4}}{2}, \frac{\sqrt{n^2-1}}{2}\big]$.
  Then, there exists a constant $\varepsilon_0>0$ such that, for every small data satisfying
  \[\|f\|_{H^1}+ \|g\|_{L^2}\leq \varepsilon\,\,\,\mbox{for}\,\,\, \varepsilon\leq \varepsilon_0,\]
  there exists a uniquely determined global (in time) energy solution \[ \phi \in C\big([0,\infty),H^1(\R^n)\big) \cap C^1\big([0,\infty),L^2(\R^n)\big).\]  Moreover, the solution $\phi$ satisfies the decay estimates \eqref{optidamp4} and \eqref{optidamp5}
  for $\gamma=1$.
  \end{theorem}
\begin{remark}
The requirement $p_{n,\mu} < p\leq \frac{n}{n-2}$  implies $\mu<2$, i.e, $m>\frac{\sqrt{ n^2-4}}{2}$ for $n\geq 3$.
\end{remark}
\begin{proof}
It is enough to prove the global existence of small data solutions to \eqref{semilieardampedwavedissipationtreatment}.
We define the  space
\[ X(t) := \bigl\{ u\in \mathcal{C}\big([0,t], H^1(\R^n) \big) \cap \mathcal{C}^1\big([0,t], L^2(\R^n) \big)\, : \ \|u\|_{X(t)} :=\sup_{\tau\in[0,t]} \{e^\tau\| u_\tau(\tau,\cdot)\|_{L^2}+ \|u(\tau,\cdot)\|_{H^1}\}<\infty \bigr\} \]
with the usual norm in $H^1(\R^n)$. For any~$u\in X(t)$ we define
\[ Pu(t,x) :=  K_0(t,0, x) \ast_{(x)} u_0(x) + K_1(t,0, x) \ast_{(x)} u_1(x)  + Gu(t,x), \]
where
\[Gu(t,x)=\int_0^t e^{(p-1)rs}K_1(t,s,x) \ast_{(x)}|u(s,x)|^{p}\,ds.\]
Using Propositions \ref{proposition1} and \ref{proposition11} for $s=0$ we have
\begin{equation}\label{eq:Kdata}
\| K_0(t,0, x) \ast_{(x)} u_0(x) + K_1(t,0, x) \ast_{(x)} u_1(x) \|_{X(t)} \lesssim\,\|u_0\|_{H^1} + \|u_1\|_{L^2}.
\end{equation}
Applying Minkowski's integral inequality and estimate \eqref{optidamp2} gives
 \[\|Gu(t,x)\|_{L^2} \lesssim \int_0^t e^{(p-1)rs}\|K_1(t,s,x) \ast_{(x)}|u(s,x)|^{p}\|_{L^2}\,ds\lesssim \int_0^t e^{(p-1)rs}
 \| |u(s,x)|^{p}\|_{L^2}\,ds.\]
 Now Gagliardo-Nirenberg inequality  comes into play.
We may estimate
\[\| |u(s,\cdot)|^{p}\|_{L^2}=\|u(s,\cdot)\|_{L^{2p}}^p\lesssim
\| u(s,\cdot)\|_{L^2}^{p(1-\theta)} \| \nabla u(s,\cdot)\|_{L^2}^{p\theta} \lesssim \|u\|_{X(s)}^p,\]
where
\[ \theta= n\Big(\frac12-\frac1{2p}\Big), \qquad 2p \leq \begin{cases}
\infty & \text{if~$n\leq 2$,}\\
\frac{2n}{n-2} & \text{if~$n\geq 3$.}
\end{cases} \]
Hence,
\[\|Gu(t,\cdot)\|_{L^2} \lesssim \|u\|_{X(t)}^p\int_0^t e^{(p-1)rs} ds\lesssim \|u\|_{X(t)}^p,\]
thanks to $r<0$ and $p>1$. Now, after using for $p >p_{n,\mu}$ the estimates \eqref{optidamp3} and \eqref{optidamp4timederivative} we may conclude
\[\|\nabla Gu(t,\cdot)\|_{L^2} \lesssim  \|u\|_{X(t)}^p \int_0^t e^{(p-1)rs +s} ds \lesssim \|u\|_{X(t)}^p, \]
and
\[ e^t\|\partial_t Gu(t,\cdot)\|_{L^2} \lesssim  \|u\|_{X(t)}^p \int_0^t e^{(p-1)rs +s} ds \lesssim \|u\|_{X(t)}^p. \]
Therefore, it follows
\begin{equation} \label{eq:mappingdissipation}
\|Pu\|_{X(t)}
     \lesssim\,\|u_0\|_{H^1} + \|u_1\|_{L^2}+ \|u\|_{X(t)}^{p}.
     \end{equation}
To derive a Lipschitz condition we recall
\begin{equation} \label{Lipschitzdissipation}
\begin{cases}
& P u - P v= G u - G v=  \int_0^t e^{(p-1)rs} K_1(t,s,x) \ast_{(x)}\big(|u(s,x)|^{p} - |v(s,x)|^p \big) ds \\
& \quad =p \int_0^t e^{(p-1)rs} K_1(t,s,x) \ast_{(x)} \Big(\int_0^1 |v + \tau (u-v)|^{p-2} (v + \tau (u-v)) d\tau\Big)(s,x) (u-v)(s,x) ds.
\end{cases}
\end{equation}
Using H\"older's inequality and Gagliardo-Nirenberg inequality we obtain
\begin{eqnarray*}
&& \|P u - P v\|_{L^2} \lesssim
\int_0^t e^{(p-1)rs} \Big(\int_0^1 \big\||v + \tau (u-v)|^{p-1}\big\|_{L^{r_1}} d\tau\Big)(s) \|(u-v)(s,\cdot)\|_{L^{r_2}} ds \\
&& \quad \lesssim  \int_0^t e^{(p-1)rs} \Big(\int_0^1 \|v + \tau (u-v)|\|^{(p-1)(1-\theta_1)}_{L^{2}} \|\nabla(v + \tau (u-v))\|^{(p-1)\theta_1}_{L^{2}}d\tau\Big)(s) \\ && \qquad \quad \times \|(u-v)(s,\cdot)\|^{1-\theta_2}_{L^{2}} \|\nabla(u-v)(s,\cdot)\|^{\theta_2}_{L^{2}}ds \\
&& \quad \lesssim \int_0^t e^{(p-1)rs} \Big(\int_0^1 \|v + \tau (u-v)|\|^{p-1}_{X(s)} d\tau\Big)(s) \|u-v\|_{X(s)} ds \\
&& \quad \lesssim \|u-v\|_{X(t)} \bigl(\|u\|_{X(t)}^{p-1}+\|v\|_{X(t)}^{p-1}\bigr).
\end{eqnarray*}
Here we have chosen $r_1$ and $r_2$ in such a way that  $\frac{1}{r_1} + \frac{1}{r_2}=\frac{1}{2}$, $\theta_1=n(\frac{1}{2} - \frac{1}{r_1(p-1)}) \in [0,1]$ and $\theta_2=n(\frac{1}{2}-\frac{1}{r_2}) \in [0,1]$. If we choose the parameters $r_1=\frac{2p}{p-1}>2$ and $r_2=2p>2$, then we can verify all these conditions for $n \geq 2$.
In the same manner we are able to prove
\begin{eqnarray*}
\|\nabla(P u - P v)\|_{L^2} \lesssim \|u-v\|_{X(t)} \bigl(\|u\|_{X(t)}^{p-1}+\|v\|_{X(t)}^{p-1}\bigr),\,\,\,
e^t \| \partial_t(P u - P v)\|_{L^2} \lesssim \|u-v\|_{X(t)} \bigl(\|u\|_{X(t)}^{p-1}+\|v\|_{X(t)}^{p-1}\bigr)
\end{eqnarray*}
for the admissible range of $p$.\\
Summarizing all the estimates we have
     \begin{equation}
\label{eq:contractiondissipation}
\|Pu-Pv\|_{X(t)}
     \lesssim \|u-v\|_{X(t)} \bigl(\|u\|_{X(t)}^{p-1}+\|v\|_{X(t)}^{p-1}\bigr)
\end{equation}
for any~$u,v\in X(t)$.
Due to (\ref{eq:mappingdissipation}) the operator $P$ maps~$X(t)$ into itself
and the existence of a
unique global solution $u$ follows by contraction (\ref{eq:contractiondissipation}) and continuation argument for small data.
Moreover, we conclude a local (in time) existence result for large data as well.
The statements of Theorem \ref{main2}, in particular, the decay estimates follow by using the relation $\phi(t,x)=e^{rt}u(t, x)$
with $r=\frac{-n+\mu}{2}$.
\end{proof}
Now, we will not require energy solutions any more, we are interested in Sobolev solutions only. We have the following result.
 \begin{theorem}\label{main20} Consider for $n \geq 2$ the Cauchy problem \eqref{general} with data $f \in H^\gamma(\R^n), \gamma \in (\frac{1}{2},1),$ and $g \in L^2(\R^n)$. Let $p>p_{n,\mu,\gamma}:=1+  \frac{2\gamma}{n-\mu}$ and $p \leq \frac{n}{n-2\gamma}$. Assume that the parameter $\mu=\sqrt{n^2-4m^2}$ satisfies
  $\mu \in [1,2\gamma)$, i.e., $m\in \big(\frac{\sqrt{n^2-4\gamma^2}}{2},\frac{\sqrt{n^2-1}}{2}\big]$.
  Then, there exists a constant $\varepsilon_0>0$ such that, for every small data satisfying
  \[\|f\|_{H^\gamma}+ \|g\|_{L^2}\leq \varepsilon\,\,\,\mbox{for}\,\,\, \varepsilon\leq \varepsilon_0,\]
  there exists a uniquely determined global (in time) Sobolev solution \[ \phi \in C\big([0,\infty),H^\gamma(\R^n)\big).\] The solution satisfies the decay estimate  \begin{equation*}
  \| \phi(t,\cdot)\|_{H^\gamma} \lesssim
  e^{\frac{(-n+\mu)t}{2}}\big(\|f\|_{H^\gamma}+ \|g\|_{L^2} \big).
 \end{equation*}
  \end{theorem}
\begin{remark}
The requirement $p_{n,\mu,\gamma} < p\leq \frac{n}{n-2\gamma}$  implies $\mu<2\gamma$, i.e, $m>\frac{\sqrt{ n^2-4\gamma^2}}{2}$ for $n\geq 2$.
\end{remark}
\begin{proof}
We only sketch the proof, in particular, the modifications to the proof of Theorem \ref{main2}. We define the  space
\[ X(t) := \bigl\{ u\in \mathcal{C}\big([0,t], H^\gamma(\R^n) \big) \, : \ \|u\|_{X(t)} :=\sup_{\tau\in[0,t]} \{\|u(\tau,\cdot)\|_{H^\gamma}\}<\infty \bigr\} \]
with the usual norm in $H^\gamma(\R^n)$. Using Propositions \ref{proposition1} and \ref{proposition11} for $s=0$ we have
\begin{equation}\label{eq:Kdatafractional}
\| K_0(t,0, x) \ast_{(x)} u_0(x) + K_1(t,0, x) \ast_{(x)} u_1(x) \|_{X(t)} \lesssim\,\|u_0\|_{H^\gamma} + \|u_1\|_{L^2}.
\end{equation}
Now fractional Gagliardo-Nirenberg inequality  comes into play.
We may estimate
\[\| |u(s,\cdot)|^{p}\|_{L^2}=\|u(s,\cdot)\|_{L^{2p}}^p\lesssim
\| u(s,\cdot)\|_{L^2}^{p(1-\theta)} \|u(s,\cdot)\|_{\dot{H}^\gamma}^{p\theta} \lesssim \|u\|_{X(s)}^p,\]
where
\[ \theta= \frac{n}{\gamma}\Big(\frac12-\frac1{2p}\Big) \in [0,1], \quad p \leq
\frac{n}{n-2\gamma}.\]
Hence,
\begin{eqnarray*}
\|Gu(t,\cdot)\|_{L^2} \lesssim \int_0^t e^{(p-1)rs}\big\|K_1(t,s,x) \ast_{(x)} |u(s,x)|^p\big\|_{L^2} ds \lesssim \int_0^t e^{(p-1)rs} \||u(s,x)|^p\|_{L^2} ds \lesssim \|u\|^p_{X(t)}
\end{eqnarray*}
thanks to the assumptions $p>1$ and $r<0$. Propositions \ref{fractionalGagliardoNirenberg} and \ref{proposition11} imply for all $h \in L^2(\mathbb{R}^n)$ the estimate
\begin{eqnarray*}
\|K_1(t,s,x) \ast_{(x)} h\|_{\dot{H}^\gamma} \leq \|K_1(t,s,x) \ast_{(x)} h\|_{\dot{H}^1}^\gamma \|K_1(t,s,x) \ast_{(x)} h\|_{L^2}^{1-\gamma} \leq e^{\gamma s} \|h\|_{L^2}.
\end{eqnarray*}
Finally, by using $p >p_{n,\mu,\gamma}$ and $r=\frac{-n + \sqrt{n^2-4 m^2}}{2}$ we may estimate
\[\|Gu(t,\cdot)\|_{\dot{H}^\gamma} \lesssim  \|u\|_{X(t)}^p \int_0^t e^{(p-1)rs +s\gamma} ds \lesssim \|u\|_{X(t)}^p. \]
Using H\"older's inequality and fractional Gagliardo-Nirenberg inequality, choosing the parameters $r_1=\frac{2p}{p-1}>2$ and $r_2=2p>2$ we can follow the steps of the proof of the Lipschitz property in the proof of Theorem \ref{main2} to obtain
     \begin{eqnarray*}
\|Pu-Pv\|_{X(t)}
     \lesssim \|u-v\|_{X(t)} \bigl(\|u\|_{X(t)}^{p-1}+\|v\|_{X(t)}^{p-1}\bigr)
\end{eqnarray*}
for any~$u,v\in X(t)$.
This completes the proof.\end{proof}
In the remaining part of this section we are interested in energy solutions having a suitable higher regularity. In the proof we will apply the
tools from the Appendix.
\begin{theorem}\label{main5} Consider the Cauchy problem  \eqref{general} with data $f \in H^\sigma(\R^n)$ and $g \in H^{\sigma-1}(\R^n)$ for $n \geq 3$, where $\sigma \in \big(1,\frac{n}{2}\big)$.
Assume that the parameter $\mu=\sqrt{n^2-4m^2}$ satisfies $\mu \in [1, 2\sigma)$, i.e.,  $m\in \big(\frac{\sqrt{n^2-4\sigma^2}}{2},\frac{\sqrt{n^2-1}}{2}\big]$.
Finally, let $p$ satisfy the following condition:
\begin{eqnarray*}  \max\{p_{n,\mu};\lceil \sigma
 \rceil\} <p \leq 1+\frac{2}{n-2\sigma}.
 \end{eqnarray*}  Then, there exists a constant $\varepsilon_0>0$ such that, for every given small data satisfying
  \[\|f\|_{H^\sigma}+ \|g\|_{H^{\sigma-1}}\leq \varepsilon\,\,\,\mbox{for}\,\,\, \varepsilon\leq \varepsilon_0,\]
  there exists a uniquely determined global (in time) energy solution \[ \phi \in C\big([0,\infty),H^\sigma(\R^n)\big) \cap C^1\big([0,\infty),H^{\sigma-1}(\R^n)\big).\] Moreover, the solution $\phi$ satisfies the decay estimates \eqref{optidamp4} and \eqref{optidamp5}
  for $\gamma=\sigma$.
  \end{theorem}
   \begin{remark}\label{newremark1}
  The assumption $\mu \in [1,2\sigma)$ implies $m\in \big(\frac{\sqrt{n^2-4\sigma^2}}{2}, \frac{\sqrt{n^2-1}}{2}\big]$. Moreover, we have to take account of  the condition  $p_{n,\mu}<1+\frac{2}{n-2\sigma}$.
The condition $p \in \big( \lceil \sigma
 \rceil,1+\frac{2}{n-2\sigma}\big]$  implies  the condition $(n-2\sigma)(\lceil \sigma
 \rceil-1)<2$.
Consequently, the  admissible interval for $p$ is not empty for $\sigma$ close to $1$  in low space dimensions, whereas higher space dimensions are allowed to suppose if $\sigma$ is close to $\frac{n}{2}$.\end{remark}
\begin{proof}
We only sketch the proof, in particular, the modifications to the proof of Theorem \ref{main2}.  It is enough to prove the global existence of small data solutions to \eqref{semilieardampedwavedissipationtreatment}.
Motivated by the estimates of Proposition \ref{proposition1} and Proposition \ref{proposition11} for $s=0$ we define for $t>0$ the scale of spaces of energy solutions with suitable regularity
\begin{eqnarray*} && X(t) := \bigl\{ u\in C\big([0,t], H^\sigma(\R^n) ) \cap C^1([0,t],H^{\sigma-1}(\R^n)\big)\\ && \quad : \, \|u\|_{X(t)} :=\sup_{\tau\in[0,t]}\bigl\{\|u(\tau,\cdot)\|_{L^2} + \||D|^\sigma u(\tau,\cdot)\|_{L^2} + e^\tau\|u_\tau(\tau,\cdot)\|_{L^2} +e^\tau\||D|^{\sigma-1}u_\tau(\tau,\cdot)\|_{L^2}\bigr\}
<\infty \bigr\}. \end{eqnarray*}
For any~$u\in X(t)$ we define
\[ Pu :=  K_0(t,0, x) \ast_{(x)} u_0(x) + K_1(t,0, x) \ast_{(x)} u_1(x)  + Gu, \]
where
\[Gu(t,x):=\int_0^t e^{(p-1) r s}K_1(t,s,x) \ast_{(x)}|u(s,x)|^{p}\,ds.\]
Using Proposition \ref{proposition1} and Proposition \ref{proposition11} for $s=0$ we have
\begin{equation*}
\| K_0(t,0, x) \ast_{(x)} u_0(x) + K_1(t,0, x) \ast_{(x)} u_1(x) \|_{X(t)} \lesssim\,\|u_0\|_{H^\sigma} + \|u_1\|_{H^{\sigma-1}}.
\end{equation*}
The estimates for $\|Gu(t,\cdot)\|_{L^2}$ and $\| \partial_t Gu(t,\cdot)\|_{L^2}$ follow as in the proof to Theorem \ref{main2}
under the restrictions $p >1+\frac{2}{n-\mu}$ and $p \leq \frac{n}{n-2\sigma}$ due to the higher regularity of the solution we can use in the Gagliardo-Nirenberg inequality. In the following we only show  how to estimate $\||D|^\sigma Gu(t,\cdot)\|_{L^2}$ and how the conditions of the theorem appear. The estimate of $\||D|^{\sigma-1} \partial_t Gu(t,\cdot)\|_{L^2}$ can be derived in an analogous way and brings no further requirements. \\
We have
\[ |D|^{\sigma} Gu(t,x)= \int_0^t e^{(p-1)r s} |D|^\sigma (|K_1(t,s,x) \ast_{(x)}|u(s,x)|^{p})\,ds.\]
The application of  Lemma \ref{Sickelauxiliary lemma} and estimate \eqref{optidamp3} yields
\begin{eqnarray*}
&& \| |D|^\sigma Gu(t,\cdot)\|_{L^2} \lesssim \int_0^t e^{(p-1)r s} \||D|^\sigma (K_1(t,s,x) \ast_{(x)}|u(s,x)|^{p})\|_{L^2}\,ds \\
&& \qquad \lesssim \int_0^t e^{(p-1)r s} \|\nabla |D|^{\sigma-1} (K_1(t,s,x) \ast_{(x)}|u(s,x)|^{p})\|_{L^2}\,ds \\&& \qquad \lesssim \int_0^t e^{(p-1)r s + s} \| |D|^{\sigma-1} |u(s,x)|^{p} \|_{L^2}\,ds.
\end{eqnarray*}
Applying Proposition \ref{Propfractionalchainrulegeneral} for $p>\lceil \sigma-1 \rceil$ and Proposition \ref{fractionalGagliardoNirenberg} we estimate $\| |D|^{\sigma-1} |u(s,\cdot)|^{p} \|_{L^2}$ as follows:
\begin{eqnarray*}
&& \| |D|^{\sigma-1} |u(s,\cdot)|^{p} \|_{L^2} \lesssim \|u\|_{L^{r_1}}^{p-1} \||D|^{\sigma-1} u \|_{L^{r_2}} \\
&& \qquad \lesssim \|u\|_{L^2}^{(p-1)(1-\theta_1)}\||D|^{\sigma}u\|_{L^2}^{(p-1)\theta_1}\|u\|_{L^2}^{1-\theta_2}\||D|^{\sigma}u\|_{L^2}^{\theta_2},
\end{eqnarray*}
where
\[ \frac{1}{2}=\frac{p-1}{r_1} + \frac{1}{r_2},\,\,\,\theta_1=\frac{n}{\sigma}\Big(\frac{1}{2} -\frac{1}{r_1}\Big)\in [0,1],\,\,\, \theta_2=\frac{n}{\sigma}\Big(\frac{1}{2} -\frac{1}{r_2}+\frac{\sigma-1}{n}\Big)\in [0,1].\]
Using $p>2$ in the first relation we have to verify
\[ \theta_1=\frac{n}{\sigma}\Big(\frac{1}{2} -\frac{1}{r_1}\Big)\in [0,1],\,\,\, \theta_2=\frac{\sigma-1}{\sigma}+ \frac{n(p-1)}{r_1\sigma}\in [0,1].\]
The first relation implies $2\leq r_1 \leq \frac{2n}{n-2\sigma}$. The condition in $\theta_2$ implies $2 \leq r_2 \leq \frac{2n}{n-2}$.
We choose the maximal value $r_2=\frac{2n}{n-2}$, so $r_1 = n(p-1)$.
The  lower and upper bound for $p$ guarantees $2\leq r_1 \leq \frac{2n}{n-2\sigma}$. Therefore, $\theta_1, \,\theta_2 \in [0,1]$ and Proposition \ref{fractionalGagliardoNirenberg} can be really applied.\\
Summarizing all the derived estimates we get
\begin{eqnarray*}
&& \| |D|^\sigma G u(t,\cdot)\|_{L^2} \lesssim \int_0^t e^{(p-1)r s + s} \| |D|^{\sigma-1} |u(s,x)|^{p} \|_{L^2}\,ds \\
&& \qquad \lesssim \int_0^t e^{(p-1) r s + s} \, ds \, \|u\|^p_{X(t)}.
\end{eqnarray*}
Hence, the estimates
\[ \|Gu\|_{X(t)} \lesssim \|u\|_{X(t)}^{p},\,\,\,\|Pu\|_{X(t)}
     \lesssim\,\|u_0\|_{H^\sigma} + \|u_1\|_{H^{\sigma-1}}+ \|u\|_{X(t)}^{p},
     \]
respectively, follow for
\[   \max\{p_{n,\mu};\lceil \sigma-1
 \rceil\} <p  \leq 1+\frac{2}{n-2\sigma}.  \]
This leads to $Pu\in X(t)$.

Now let us prove the Lipschitz property.
Due to (\ref{Lipschitzdissipation}) we have
\begin{eqnarray*}
P u - P v=p \int_0^t e^{(p-1)rs} K_1(t,s,x) \ast_{(x)} \Big(\int_0^1 |v + \tau (u-v)|^{p-2} (v + \tau (u-v)) d\tau\Big)(s,x) (u-v)(s,x) \,ds.
\end{eqnarray*}
By using the same ideas as in the proof to Theorem \ref{main2} we are able to estimate the norms $\|Pu-Pv\|_{L^2}$ and $e^t \|\partial_t(Pu -Pv)\|_{L^2}$. In the following we only show how to estimate $\||D|^\sigma (Pu-Pv)\|_{L^2}$. In the same way we may estimate
$e^t\||D|^{\sigma-1} \partial_t (Pu-Pv)\|_{L^2}$. \\
We have
\begin{eqnarray*}
&& |D|^\sigma(P u - P v)\\ && \quad =p \int_0^t e^{(p-1)rs} |D|^\sigma \Big(K_1(t,s,x) \ast_{(x)} \Big(\int_0^1 |v + \tau (u-v)|^{p-2} (v + \tau (u-v)) d\tau\Big)(s,x) (u-v)(s,x)\Big) \,ds.
\end{eqnarray*}
By Proposition \ref{proposition11} it follows
\begin{eqnarray*}
&& \|Pu-Pv\|_{\dot{H}^\sigma}= \||D|^\sigma(P u - P v)\|_{L^2}\\ && \quad \lesssim \int_0^t e^{(p-1)rs+s} \Big\|\Big(\int_0^1 |v + \tau (u-v)|^{p-2} (v + \tau (u-v)) d\tau\Big)(s,x) (u-v)(s,x)\Big)\Big\|_{\dot{H}^{\sigma-1}} (s) \,ds.
\end{eqnarray*}
The application of the fractional Leibniz rule from Proposition \ref{fractionalLeibniz} yields
\begin{eqnarray*}
&& \|Pu-Pv\|_{\dot{H}^\sigma}\\ && \quad \lesssim \int_0^t e^{(p-1)rs+s} \||D|^{\sigma-1}(u-v)(s,\cdot)\|_{L^{r_2}} \int_0^1 \big\||v + \tau (u-v)|^{p-2} (v + \tau (u-v)) \big\|_{L^{r_1}} d\tau \,ds\\
&& \quad \quad + \int_0^t e^{(p-1)rs+s} \|(u-v)(s,\cdot)\|_{L^{r_4}} \int_0^1 \big\||D|^{\sigma-1}\big(|v + \tau (u-v)|^{p-2} (v + \tau (u-v))\big) \big\|_{L^{r_3}} d\tau \,ds,
\end{eqnarray*}
under the conditions
\[ \frac{1}{r_1} + \frac{1}{r_2}=\frac{1}{r_3} + \frac{1}{r_4}=\frac{1}{2}.\]
Now let us estimate all the terms appearing in the above integrals. By using Proposition \ref{fractionalGagliardoNirenberg}
we arrive at the estimate
\begin{eqnarray*}
&& \big\||v + \tau (u-v)|^{p-2} (v + \tau (u-v)) \big\|_{L^{r_1}}=\|v + \tau (u-v)\|^{p-1}_{L^{r_1(p-1)}} \\
&& \qquad \lesssim \|v + \tau (u-v)\|^{(1-\theta_1)(p-1)}_{L^{2}}\big\||D|^\sigma(v + \tau (u-v))\big\|^{\theta_1(p-1)}_{L^{2}}
\end{eqnarray*}
under the condition
\[ 0 \leq \theta_1=\frac{n}{\sigma}\Big(\frac{1}{2} - \frac{1}{r_1(p-1)}\Big)\leq 1, \,\,\,\mbox{that is},\,\,\,2 \leq r_1(p-1)\leq \frac{2n}{n-2\sigma}.\]
By using Proposition \ref{fractionalGagliardoNirenberg} we get for the second term
\[     \||D|^{\sigma-1}(u-v)\|_{L^{r_2}} \lesssim    \|u-v\|^{1-\theta_2}_{L^{2}}\big\||D|^\sigma(u-v)\big\|^{\theta_2}_{L^{2}}\]
under the condition
\[ \frac{\sigma-1}{\sigma} \leq \theta_2=\frac{n}{\sigma}\Big(\frac{1}{2} - \frac{1}{r_2}+\frac{\sigma-1}{n}\Big)\leq 1, \,\,\,\mbox{that is},\,\,\,2 \leq r_2\leq \frac{2n}{n-2}.\]
In the same way we estimate the fourth term
\[     \|u-v\|_{L^{r_4}} \lesssim    \|u-v\|^{1-\theta_4}_{L^{2}}\big\||D|^\sigma(u-v)\big\|^{\theta_4}_{L^{2}}\]
under the condition
\[ 0 \leq \theta_4=\frac{n}{\sigma}\Big(\frac{1}{2} - \frac{1}{r_4}\Big)\leq 1, \,\,\,\mbox{that is},\,\,\,2 \leq r_4\leq \frac{2n}{n-2\sigma}.\]
To estimate the third term we apply Proposition \ref{Propfractionalchainrulegeneral}. In this way we obtain
\begin{eqnarray*}
\big\||D|^{\sigma-1}\big(|v + \tau (u-v)|^{p-2} (v + \tau (u-v))\big) \big\|_{L^{r_3}} \lesssim \|v + \tau (u-v)\|^{p-2}_{L^{r_5}}
\||D|^{\sigma-1}(v + \tau (u-v))\|_{L^{r_6}}
\end{eqnarray*}
under the conditions
\[ \frac{1}{r_6} + \frac{p-2}{r_5}=\frac{1}{r_3},\,\,\, p-1 > \lceil \sigma-1
 \rceil,\,\, p >\lceil \sigma
 \rceil,\]
respectively. Finally, after application of Proposition \ref{fractionalGagliardoNirenberg} the following estimates follow:
\[
\|v + \tau (u-v)\|^{p-2}_{L^{r_5}} \lesssim \|v + \tau (u-v)\|^{(1-\theta_5)(p-2)}_{L^{2}} \||D|^\sigma (v + \tau (u-v))\|^{\theta_5(p-2)}_{L^{2}}
\]
under the condition
\[   0 \leq \theta_5=\frac{n}{\sigma}\Big(\frac{1}{2} - \frac{1}{r_5}\Big)\leq 1, \,\,\,\mbox{that is},\,\,\,2 \leq r_5\leq \frac{2n}{n-2\sigma},                                                \]
and
\[
\||D|^{\sigma-1}(v + \tau (u-v))\|_{L^{r_6}} \lesssim \|v + \tau (u-v)\|^{1-\theta_6}_{L^{2}} \||D|^\sigma (v + \tau (u-v))\|^{\theta_6}_{L^{2}}
\]
under the condition
\[ \frac{\sigma-1}{\sigma} \leq \theta_6=\frac{n}{\sigma}\Big(\frac{1}{2} - \frac{1}{r_6}+\frac{\sigma-1}{n}\Big)\leq 1, \,\,\,\mbox{that is},\,\,\,2 \leq r_6\leq \frac{2n}{n-2}.\]
It remains to verify a suitable choice of parameters $r_1$ to $r_6$ to verify all the assumptions of the theorem.
We choose $r_2=\frac{2n}{n-2}>2$. So, $r_1=n$. The condition
\[ r_1(p-1)\leq \frac{2n}{n-2\sigma}\,\,\,\mbox{implies}\,\,\,p \leq 1+\frac{2}{n-2\sigma},  \]
this is one of the assumptions of the theorem. We choose $r_4=n(p-1)>2$ and $r_6=\frac{2n}{n-2}$. Then $r_3=\frac{2n(p-1)}{n(p-1)-2} >2$ and $r_5=n(p-1)>2$.
It remains to verify $r_5 \leq \frac{2n}{n-2\sigma}$. But this gives $p \leq 1+\frac{2}{n-2\sigma}$.
\noindent Summarizing
we have proved \begin{equation*}
\|P u-P v\|_{\dot{H}^{\sigma}}
     \lesssim \|u-v\|_{X(t)} \bigl(\|u\|_{X(t)}^{p-1}+\|v\|_{X(t)}^{p-1}\bigr)
\end{equation*}
for any~$u,v\in X(t)$.
In the same way we estimate $\|\partial_t(P u-P v)\|_{\dot{H}^{\sigma-1}}$ with no other requirements to the admissible exponents $p$.
All the derived estimates yield
     \begin{equation*}
\|Pu-Pv\|_{X(t)}
     \lesssim \|u-v\|_{X(t)} \bigl(\|u\|_{X(t)}^{p-1}+\|v\|_{X(t)}^{p-1}\bigr)
\end{equation*}
for any~$u,v\in X(t)$.
Consequently, the operator $P$ maps $X(t)$ into itself
and the existence of a
uniquely determined global (in time) energy solution $u$ with suitable higher regularity follows by the Lipschitz property and by  a continuation argument for small data.
Moreover, we conclude a local (in time) existence result for large data as well.
The decay estimates of Theorem \ref{main5} follow by using the relation $\phi(t,x)=e^{r t}u(t,x)$.
\end{proof}
If $n \leq 2\sigma$, then we can follow the steps of the previous proof and get the following statement.
\begin{corollary} \label{Corlargeregularity1}
Consider the Cauchy problem  \eqref{general} with data $f \in H^{\sigma}(\R^n)$ and $g \in H^{\sigma-1}(\R^n)$ for $n \geq 3$ and
$n \leq 2\sigma$.
Assume that the parameter $\mu=\sqrt{n^2-4m^2}$ satisfies $\mu \in [1,n)$.
Finally, let $p$ satisfy the following condition:
\begin{eqnarray*}  && \max\big\{p_{n,\mu};\big\lceil \sigma
 \big\rceil\big\} <p <\infty.
 \end{eqnarray*}  Then, there exists a constant $\varepsilon_0>0$ such that, for every given small data satisfying
 \[\|f\|_{H^{\sigma}}+ \|g\|_{H^{\sigma-1}}\leq \varepsilon\,\,\,\mbox{for}\,\,\, \varepsilon\leq \varepsilon_0,\]
  there exists a uniquely determined global (in time) energy solution \[ \phi \in C\big([0,\infty),H^{\sigma}(\R^n)\big) \cap C^1\big([0,\infty),H^{\sigma-1}(\R^n)\big).\] Moreover, the solution $\phi$ satisfies the decay estimates \eqref{optidamp4} and \eqref{optidamp5}
  for $\gamma=\sigma$.
\end{corollary}
 \begin{remark}
 Under the assumption $n \leq 2\sigma$  we do not have longer an upper bound for $p$ as we have in  Theorem \ref{main5} for $\sigma\in (1, \frac{n}{2})$.
This implies an essential difference to the case $n > 2\sigma$ (see Remark \ref{newremark1}). Now we may avoid a positive lower  bound for $m^2$. The range of admissible space dimensions increases with $\sigma$.
 \end{remark}
 For  $\sigma > \frac{n}{2}$, thanks to Proposition \ref{PropSickelfractional} and  Sobolev's embedding theorem $\| u(s,\cdot)\|_{L^{\infty}}\lesssim  \| u(s,\cdot)\|_{H^{\sigma}(\R^n)}$,
 we may improve the lower bound for $p$ in Corollary \ref{Corlargeregularity1}.
 \begin{theorem}\label{improve1} Consider the Cauchy problem  \eqref{general} with data $(f,g) \in (H^\sigma(\R^n)\times H^{\sigma-1}(\R^n))$ for
 $\sigma > \frac{n}{2} $ and $n \geq 2$.  Assume that the parameter $\mu=\sqrt{n^2-4m^2}$ satisfies $\mu \in [1,n)$, i.e.,  $m\in (0, \frac{\sqrt{n^2-1}}{2}]$.
 Finally, let $p$ satisfy the following condition:
\begin{eqnarray*}  && \max\big\{p_{n,\mu}; \sigma; 2\big\} <p <\infty.
 \end{eqnarray*}
  Then, there exists a constant $\varepsilon_0>0$ such that, for every given small data satisfying
  \[\|f\|_{H^\sigma}+ \|g\|_{ H^{\sigma-1}} \leq \varepsilon\,\,\,\mbox{for}\,\,\, \varepsilon\leq \varepsilon_0,\]
 there exists a uniquely determined global (in time) energy solution \[ \phi \in C\big([0,\infty),H^\sigma(\R^n)\big)\cap C^1\big([0,\infty), H^{\sigma-1}(\R^n)\big).\]  Moreover, the solution $\phi$ satisfies the decay estimate \[  \|\phi(t,\cdot)\|_{H^\sigma} + \|\phi_t(t,\cdot)\|_{H^{\sigma-1}}  \lesssim  e^{\frac{(-n+\mu)t}{2}}\big(\|f\|_{H^\sigma}+ \|g\|_{H^{\sigma-1}} \big). \]
 \end{theorem}
 \begin{proof}
Here we only sketch the differences to the proof of Theorem \ref{main5}. Let $u\in X(t)$ with $X(t)$ defined as in Theorem \ref{main5}.
It is clear that $u(t,\cdot)\in H^{\sigma}(\R^n)\subset L^{\infty}(\R^n)$.

By using Corollary \ref{Corfractionalhomogeneous}  we may estimate for $p> \sigma-1$
\[\| |D|^{\sigma -1} |u(s,\cdot)|^{p}\|_{ L^2 } \lesssim \|  u(s,\cdot)\|_{\dot{H}^{\sigma-1}} \|u(s,\cdot)\|_{L^\infty}^{p-1} \lesssim \|u\|_{X(s)}^p,\]
and conclude
\[\||D|^{\sigma-1}  Gu(t,\cdot)\|_{L^2} \lesssim \|u\|_{X(t)}^p\]
for all $p> \max\{p_{n,\mu}; \sigma-1\}.$

The application of fractional Leibniz rule from Proposition \ref{Propfractionalchainrulegeneral} yields
\begin{eqnarray*}
&& \|Pu-Pv\|_{\dot{H}^\sigma}\\ && \quad \lesssim \int_0^t e^{(p-1)rs+s} \||D|^{\sigma-1}(u-v)(s,\cdot)\|_{L^{2}} \int_0^1 \big\||v + \tau (u-v)|^{p-2} (v + \tau (u-v)) \big\|_{L^{\infty}} d\tau \,ds\\
&& \quad \quad + \int_0^t e^{(p-1)rs+s} \|(u-v)(s,\cdot)\|_{L^{\infty}} \int_0^1 \big\||D|^{\sigma-1}\big(|v + \tau (u-v)|^{p-2} (v + \tau (u-v))\big) \big\|_{L^{2}} d\tau \,ds.
\end{eqnarray*}
Putting $w=v + \tau (u-v)$ and applying  Corollary \ref{Corfractionalhomogeneous} for $p> \max\{\sigma, 2\}$ we get
\[\||D|^{\sigma-1}|w|^{p-2}w  \|_{L^{2}}\lesssim  \|  w\|_{\dot{H}^{\sigma-1}}\| w \|_{L^{\infty}}^{p-2},\]
and thanks to  Sobolev's embedding theorem $\| w(s,\cdot)\|_{L^{\infty}}\lesssim  \| w(s,\cdot)\|_{H^{\sigma}}$ we conclude
for all $p> \max\{p_{n,\mu}; 2; \sigma\}$
\[\|P u-P v\|_{\dot{H}^{\sigma}}
     \lesssim \|u-v\|_{X(t)} \bigl(\|u\|_{X(t)}^{p-1}+\|v\|_{X(t)}^{p-1}\bigr)
\]
for any $u,v\in X(t)$.
\end{proof}
 If $ p_{n,\mu}>\sigma + 1 $, then we may improve the lower bound for $p$ in Theorem \ref{improve1}.
\begin{theorem}\label{main1} Consider the Cauchy problem  \eqref{general} with data $(f,g) \in (H^\sigma(\R^n)\times H^{\sigma-1}(\R^n))$ for
 $\sigma > \frac{n}{2} $ and $n \geq 2$. Let $p>\sigma +1$. Assume that the parameter $\mu=\sqrt{n^2-4m^2}$ satisfies $\mu \in [1,n)$, i.e.,  $m\in (0, \frac{\sqrt{n^2-1}}{2}]$.
  Then, there exists a constant $\varepsilon_0>0$ such that, for every given small data satisfying
  \[\|f\|_{H^\sigma}+ \|g\|_{ H^{\sigma-1}} \leq \varepsilon\,\,\,\mbox{for}\,\,\, \varepsilon\leq \varepsilon_0,\]
 there exists a uniquely determined global (in time) energy solution \[ \phi \in C\big([0,\infty),H^\sigma(\R^n)\big)\cap C^1\big([0,\infty), H^{\sigma-1}(\R^n)\big).\]  Moreover, the solution $\phi$ satisfies the decay estimate \[  \|\phi(t,\cdot)\|_{H^\sigma} + \|\phi_t(t,\cdot)\|_{H^{\sigma-1}}  \lesssim  e^{\frac{(-n+\mu)t}{2}}\big(\|f\|_{H^\sigma}+ \|g\|_{H^{\sigma-1}} \big). \]
 \end{theorem}
 \begin{remark}
In Theorem \ref{main1}, the solution  $\phi$
  satisfies the same decay estimate as the solution to the corresponding linear Cauchy problem.
   In Theorem 0.1 of \cite{Y} we have a loss of decay.
\end{remark}
\begin{proof}
In the proof of Theorem \ref{main1} we shall use Proposition \ref{PropSickelfractional} and Corollary \ref{Corfractionalhomogeneous}.\\
Here we only sketch the differences to the proof of Theorem \ref{main2}.
We define the  space
\begin{eqnarray*} && X(t) := \bigl\{ u\in \mathcal{C}\big([0,t], H^\sigma(\R^n) \big) \cap \mathcal{C}^1\big([0,t], H^{\sigma-1}(\R^n) \big)\,\\ && \qquad  : \ \|u\|_{X(t)} :=\sup_{\tau\in[0,t]} \{    \| u_\tau(\tau,\cdot)\|_{H^{{\sigma-1}}}
+ \| u(\tau,\cdot)\|_{{H^{\sigma}}}  \}<\infty \bigr\}
\end{eqnarray*}
with the usual norm in $H^{\sigma}(\mathbb{R}^n)$. After using Propositions \ref{proposition1} and \ref{proposition11} for $s=0$ we have for any $u\in X(t)$
the estimate
\begin{equation}\label{eq:Kdata11}
\| K_0(t,0, x) \ast_{(x)} u_0(x) + K_1(t,0, x) \ast_{(x)} u_1(x) \|_{X(t)} \lesssim\,\|u_0\|_{H^{\sigma}} + \|u_1\|_{H^{\sigma-1}}.
\end{equation}
For $j=0,1$  Minkowski's integral inequality implies
 \[ \|\partial_t^j Gu(t,x)\|_{L^2} \lesssim \int_0^t e^{(p-1)rs} \|\partial_t^j K_1(t,s,x) \ast_{(x)}|u(s,x)|^{p}\|_{L^2}\,ds\lesssim
 \int_0^t e^{(p-1)rs}
 \| |u(s,x)|^{p}\|_{L^2}\,ds.\]
Now we may estimate $\| |u(s,\cdot)|^{p}\|_{L^2}=\| u(s,\cdot)\|_{L^{2p}}^p  \lesssim \|u\|_{X(s)}^p$ and thanks to $r<0$ we may conclude for all $p>1$ the estimate
 \[ \|\partial_t^j Gu(t,x)\|_{L^2} \lesssim \|u\|_{X(s)}^p.\]
 Now, under the assumption $p>\sigma>\frac{n}{2}$,  thanks to Proposition \ref{PropSickelfractional}
 we may directly estimate $\| |D|^\sigma |u(s,\cdot)|^{p}\|_{ L^2 }$ in terms of $\|  u(s,\cdot)\|_{H^{\sigma}}$. So, since we do no longer need to apply Gagliardo-Nirenberg inequality as done in the proof of Theorems \ref{main2} and \ref{main5}, we may change the arguments to estimate $\||D|^\sigma Gu(t,x)\|_{L^2}$ and $\||D|^{\sigma-1}\partial_t Gu(t,\cdot)\|_{L^2}$.
 Indeed,  using the estimate \eqref{optidamp2}  gives
  \begin{eqnarray*}
  && \||D|^\sigma Gu(t,x)\|_{L^2} \lesssim  \int_0^t e^{(p-1)rs} \||D|^\sigma K_1(t,s,x) \ast_{(x)}|u(s,x)|^{p}\|_{L^2}\,ds \\
  && \qquad \lesssim
 \int_0^t e^{(p-1)rs}
 \| |D|^\sigma |u(s,x)|^{p}\|_{L^2}\,ds.
 \end{eqnarray*}
Moreover, thanks to  Sobolev's embedding theorem $\| u(s,\cdot)\|_{L^{\infty}}\lesssim  \| u(s,\cdot)\|_{H^{\sigma}}$  and by using Proposition \ref{PropSickelfractional}  we may estimate for $p> \sigma$
\[\| |D|^\sigma |u(s,\cdot)|^{p}\|_{ L^2 } \lesssim \|  u(s,\cdot)\|_{H^{\sigma}} \|u(s,\cdot)\|_{L^\infty}^{p-1} \lesssim \|u\|_{X(s)}^p.\]
 Hence, for $p>\sigma$ it follows
\[\|Gu(t,\cdot)\|_{H^{\sigma}} \lesssim \|u\|_{X(t)}^p\int_0^t  e^{(p-1)rs} ds\lesssim \|u\|_{X(t)}^p.\]
Similarly,
by applying Corollary \ref{Corfractionalhomogeneous} with $p>\sigma-1$
we may estimate
\[\| |D|^{\sigma-1} |u(s,\cdot)|^{p}\|_{L^2}\lesssim  \| |D|^{\sigma-1} u(s,\cdot)\|_{L^2}\| u(s,\cdot)\|_{L^{\infty}}^{p-1}.\]
Using  that $e^{s-t}$ is bounded in estimate \eqref{optidamp4timederivative}  we conclude
\begin{eqnarray*}
&&\| |D|^{\sigma-1} |u(s,\cdot)|^{p}\|_{L^2}\lesssim  \| |D|^{\sigma-1} u(s,\cdot)\|_{L^2}\| u(s,\cdot)\|_{L^{\infty}}^{p-1}\lesssim \|u(s,\cdot)\|_{H^\sigma}^p, \\
&& \||D|^{\sigma-1}\partial_t Gu(t,\cdot)\|_{L^2} \lesssim  \|u\|_{X(t)}^p \int_0^t e^{(p-1)rs} ds \lesssim \|u\|_{X(t)}^p \end{eqnarray*}
due to $r<0$ and $p >1$, respectively.

Now we have to estimate $\|P u - P v\|_{H^\sigma}$ and $\|\partial_t(P u - P v)\|_{H^{\sigma-1}}$.
Due to (\ref{Lipschitzdissipation}) we have
\begin{eqnarray*}
P u - P v=p \int_0^t e^{(p-1)rs} K_1(t,s,x) \ast_{(x)} \Big(\int_0^1 |v + \tau (u-v)|^{p-2} (v + \tau (u-v)) d\tau\Big)(s,x) (u-v)(s,x) \,ds.
\end{eqnarray*}
By using  \eqref{optidamp2} and  that $H^{\sigma}(\R^n)$ is an algebra for $\sigma>\frac{n}{2}$ we get
\[ \|P u - P v)\|_{H^\sigma}\lesssim \int_0^t e^{(p-1)rs} \Big(\int_0^1 \big\||v + \tau (u-v)|^{p-2} (v + \tau (u-v))\big\|_{H^\sigma} d\tau\Big)(s) \|(u-v)(s,\cdot)\|_{H^\sigma} \,ds.
\]
Under the assumption $p-1>\sigma$ it  follows from Proposition \ref{PropSickelfractional} that
\begin{eqnarray*}
\Big\||v + \tau (u-v)|^{p-2} (v + \tau (u-v)) \Big\|_{H^{\sigma}}
 \lesssim \Big\|v + \tau (u-v)\Big\|_{H^{\sigma}}  \Big\|v + \tau (u-v)) \Big\|_{L^{\infty}}^{p-2}.
\end{eqnarray*}
Due  to $p>\sigma+1$, it  follows from  Sobolev's embedding theorem that
\begin{eqnarray*}
&& \|P u - P v\|_{H^\sigma}\lesssim  \int_0^t e^{(p-1)rs}
\Big\|(v + \tau (u-v))\Big\|_{H^{\sigma}}^{p-1}\|u-v\|_{X(s)} ds \\
&& \qquad \lesssim \big(\|u\|_{X(t)}^{p-1} + \|v\|_{X(t)}^{p-1}\big)\|u-v\|_{X(t)}\int_0^t e^{(p-1)rs}  ds \lesssim \|u-v\|_{X(t)}\big(\|u\|_{X(t)}^{p-1} + \|v\|_{X(t)}^{p-1}\big).
\end{eqnarray*}
In the same way we may conclude as above
\begin{eqnarray*}
&& \|\partial_t(P u - P v)\|_{H^{\sigma-1}}
  \lesssim \big(\|u\|_{X(t)}^{p-1} + \|v\|_{X(t)}^{p-1}\big)\|u-v\|_{X(t)}\int_0^t e^{(p-1)rs}  ds \lesssim \|u-v\|_{X(t)}\big(\|u\|_{X(t)}^{p-1} + \|v\|_{X(t)}^{p-1}\big).
\end{eqnarray*}
Summarizing we have
     \begin{equation*}
\|Pu-Pv\|_{X(t)}
     \lesssim \|u-v\|_{X(t)} \bigl(\|u\|_{X(t)}^{p-1}+\|v\|_{X(t)}^{p-1}\bigr)
\end{equation*}
for any~$u,v\in X(t)$.
Consequently, the operator $P$ maps $X(t)$ into itself.
and the existence of a
uniquely determined global (in time) energy solution $u$ follows by the Lipschitz property and by  a continuation argument for small data.
Moreover, we conclude a local (in time) existence result for large data as well.
The decay estimates of Theorem \ref{main1} follow by using the relation $\phi(t,x)=e^{r t}u(t,x)$.
This completes the proof.
\end{proof}

\subsection{Global existence of small data solutions: Case $m\in \big(\frac{\sqrt{n^2-1}}{2}, \frac{n}{2}\big)$.} \label{Sec3.3}

 In this section we consider the case of a non-effective dissipation in \eqref{semilieardampedwavedissipationtreatment}, i.e., the parameter $\mu=\sqrt{n^2-4m^2}$ satisfies $\mu \in (0,1)$, that implies $m\in \big(\frac{\sqrt{n^2-1}}{2},\frac{n}{2}\big)$.
 In general, one might expect additional restrictions on the power non-linearity for non-effectively damped models in comparison with effectively damped models.
 But this will be not the case for the models we shall treat in this section. The reason is that the exponential function at the right-hand side of  \eqref{semilieardampedwavedissipationtreatment} decays faster if  $\mu $ becomes smaller.
\begin{theorem}\label{noneffective}  Consider for $n \geq 2$ the Cauchy problem \eqref{general} with
$m\in \big(\frac{\sqrt{n^2-1}}{2}, \frac{n}{2}\big)$ and
data $f \in H^1(\R^n)$ and $g \in L^2(\R^n)$. Let $p>\frac{n+1}{n-1}$  and $p\leq \frac{n}{n-2}$ for $n\geq 3$.
Then, there exists a constant $\varepsilon_0>0$ such that, for every small data satisfying
  \[\|f\|_{H^1}+ \|g\|_{L^2}\leq \varepsilon\,\,\,\mbox{for}\,\,\, \varepsilon\leq \varepsilon_0,\]
  there exists a uniquely determined global (in time) energy solution \[ \phi \in C\big([0,\infty),H^1(\R^n)\big) \cap C^1\big([0,\infty),L^2(\R^n)\big).\]  Moreover, the solution $\phi$ satisfies the estimates \eqref{optidamp3energy} and \eqref{optidamp4energy}
  for $\gamma=1$.
  \end{theorem}
  \begin{remark}
If we compare the range of admissible powers $p$ in Theorems \ref{main2} and \ref{noneffective} we have that $p_{n,\mu}=1+  \frac{2}{n-\mu}> 1+  \frac{2}{n-1}=\frac{n+1}{n-1}$ for all $\mu\in [1,n)$ and $n\geq 2$.
\end{remark}
\begin{proof}
It is enough to prove the global existence of small data solutions to \eqref{semilieardampedwavedissipationtreatment}.
We define the  space
\begin{eqnarray*} && X(t) := \bigl\{ u\in \mathcal{C}\big([0,t], H^1(\R^n) \big) \cap \mathcal{C}^1\big([0,t], L^2(\R^n) \big)\, \\ && \qquad : \ \|u\|_{X(t)} :=\sup_{\tau\in[0,t]} \big\{e^{\frac{(\mu-1)\tau}{2}}\|u(\tau,\cdot)\|_{H^1}+ e^{\mu\tau}\| u_\tau(\tau,\cdot)\|_{L^2}\big\}<\infty \bigr\}. \end{eqnarray*}
For any $u\in X(t)$ we define
\[ Pu(t,x) :=  K_0(t,0, x) \ast_{(x)} u_0(x) + K_1(t,0, x) \ast_{(x)} u_1(x)  + Gu(t,x), \]
where
\[Gu(t,x)=\int_0^t e^{(p-1)rs}K_1(t,s,x) \ast_{(x)}|u(s,x)|^{p}\,ds.\]
Using Propositions \ref{proposition1} and \ref{proposition11} for $s=0$ we have
\begin{equation}\label{eq:Kdatanoneffective}
\| K_0(t,0, x) \ast_{(x)} u_0(x) + K_1(t,0, x) \ast_{(x)} u_1(x) \|_{X(t)} \lesssim\,\|u_0\|_{H^1} + \|u_1\|_{L^2}.
\end{equation}
 Using Minkowski's integral inequality and \eqref{optidamp2} gives
 \[\|Gu(t,x)\|_{L^2} \lesssim \int_0^t e^{(p-1)rs}\|K_1(t,s,x) \ast_{(x)}|u(s,x)|^{p}\|_{L^2}\,ds\lesssim \int_0^t e^{(p-1)rs}
 \| |u(s,x)|^{p}\|_{L^2}\,ds.\]
Now Gagliardo-Nirenberg inequality  comes into play.
We may estimate
\[\| |u(s,\cdot)|^{p}\|_{L^2}=\|u(s,\cdot)\|_{L^{2p}}^p\lesssim
\| u(s,\cdot)\|_{L^2}^{p(1-\theta)} \| \nabla u(s,\cdot)\|_{L^2}^{p\theta} \lesssim e^{\frac{(1-\mu)ps}{2}}\|u\|_{X(s)}^p,\]
where
\begin{equation}\label{thetaNone}
 \theta= n\Big(\frac12-\frac1{2p}\Big), \qquad 2p \leq \begin{cases}
\infty & \text{if~$n\leq 2$,}\\
\frac{2n}{n-2} & \text{if~$n\geq 3$.}
\end{cases}
\end{equation}
Hence,
\[e^{\frac{(\mu-1)t}{2}}\|Gu(t,\cdot)\|_{L^2} \lesssim \|u\|_{X(t)}^p\int_0^t e^{(p-1)\left(r+\frac{1-\mu}{2}\right)s} ds\lesssim
 \|u\|_{X(t)}^p,\]
thanks to $r+\frac{1-\mu}{2}<0$ and $p>1$. Now, after using \eqref{optidampwithloss} and
  \eqref{optidamp3timederivative} we obtain for $p>\frac{n+1}{n-1}$ the estimate
\[e^{\frac{(\mu-1)t}{2}}\|\nabla Gu(t,\cdot)\|_{L^2} \lesssim  \|u\|_{X(t)}^p \int_0^t e^{(p-1)rs +\frac{(\mu+1)s}{2}+\frac{(1-\mu)ps}{2}} ds
 \lesssim  \|u\|_{X(t)}^p \int_0^t e^{(p-1)\left(r+\frac{1-\mu}{2}\right)s +s} ds\lesssim \|u\|_{X(t)}^p, \]
and
\[ e^{\mu t}\|\partial_t Gu(t,\cdot)\|_{L^2} \lesssim  \|u\|_{X(t)}^p \int_0^t e^{(p-1)rs +\mu s +\frac{(1-\mu)ps}{2}} ds \lesssim \|u\|_{X(t)}^p. \]
Therefore, it follows
\begin{equation} \label{eq:mappingdissipationnon-effective}
\|Pu\|_{X(t)}
     \lesssim\,\|u_0\|_{H^1} + \|u_1\|_{L^2}+ \|u\|_{X(t)}^{p}.
     \end{equation}
To derive a Lipschitz condition we recall
\begin{equation} \label{LipschitzdissipationNone}
\begin{cases}
& P u - P v= G u - G v=  \int_0^t e^{(p-1)rs} K_1(t,s,x) \ast_{(x)}\big(|u(s,x)|^{p} - |v(s,x)|^p \big)\, ds \\
& \quad =p \int_0^t e^{(p-1)rs} K_1(t,s,x) \ast_{(x)} \big(\int_0^1 |v + \tau (u-v)|^{p-2} (v + \tau (u-v)) d\tau\big)(s,x) (u-v)(s,x)\, ds.
\end{cases}
\end{equation}
Thanks to Proposition \ref{proposition11} and using H\"older's inequality  we get
\[  \|P u - P v\|_{L^2} \lesssim
\int_0^t e^{(p-1)rs} \Big(\int_0^1 \big\||v + \tau (u-v)|^{p-1}\big\|_{L^{2p'}} d\tau\Big)(s) \|(u-v)(s,\cdot)\|_{L^{2p}} \,ds, \]
where $p'=\frac{p}{p-1}$. Hence,
\[ \big\||v + \tau (u-v)|^{p-1}\big\|_{L^{2p^{'}}}=\big\|v + \tau (u-v)|\big\|_{L^{2p}}^{p-1}.\]
 Applying Gagliardo-Nirenberg inequality yields
\begin{eqnarray*}
&& e^{\frac{(\mu-1)t}{2}}\|P u - P v\|_{L^2}  \lesssim e^{\frac{(\mu-1)t}{2}} \int_0^t e^{(p-1)rs} \Big(\int_0^1 \|v + \tau (u-v)|\|^{(p-1)(1-\theta)}_{L^{2}} \|\nabla(v + \tau (u-v))\|^{(p-1)\theta}_{L^{2}}d\tau\Big)(s) \\ && \qquad \quad \times \|(u-v)(s,\cdot)\|^{1-\theta}_{L^{2}} \|\nabla(u-v)(s,\cdot)\|^{\theta}_{L^{2}}ds \\
&& \quad \lesssim \|u-v\|_{X(t)} \bigl(\|u\|_{X(t)}^{p-1}+\|v\|_{X(t)}^{p-1}\bigr)\int_0^t e^{(p-1)\left(r+\frac{1-\mu}{2}\right)s} ds \\
&& \quad \lesssim \|u-v\|_{X(t)} \bigl(\|u\|_{X(t)}^{p-1}+\|v\|_{X(t)}^{p-1}\bigr),
\end{eqnarray*}
where $\theta$ is given by \eqref{thetaNone}.
In the same manner we are able to prove
\begin{eqnarray*}
&e^{\frac{(\mu-1)t}{2}}\|\nabla(P u - P v)\|_{L^2} \lesssim \|u-v\|_{X(t)} \bigl(\|u\|_{X(t)}^{p-1}+\|v\|_{X(t)}^{p-1}\bigr),\\
&e^{\mu t} \| \partial_t(P u - P v)\|_{L^2} \lesssim \|u-v\|_{X(t)} \bigl(\|u\|_{X(t)}^{p-1}+\|v\|_{X(t)}^{p-1}\bigr)
\end{eqnarray*}
for the admissible range of $p$. Summarizing all these estimates we have
     \begin{equation}
\label{eq:contractiondissipationnon-effective}
\|Pu-Pv\|_{X(t)}
     \lesssim \|u-v\|_{X(t)} \bigl(\|u\|_{X(t)}^{p-1}+\|v\|_{X(t)}^{p-1}\bigr)
\end{equation}
for any $u,v\in X(t)$.
Due to (\ref{eq:mappingdissipationnon-effective}) the operator $P$ maps $X(t)$ into itself
and the existence of a
uniquely determined global (in time) solution $u$ follows by contraction (\ref{eq:contractiondissipationnon-effective}) and continuation argument for small data.
Moreover, we conclude a local (in time) existence result for large data as well.
The statement of Theorem \ref{noneffective} follows by using the relation $\phi(t,x)=e^{rt}u(t, x)$
with $r=\frac{-n+\mu}{2}$.
\end{proof}
Now, we will not require energy solutions any more, we are interested in Sobolev solutions only. We have the following result.
 \begin{theorem}\label{noneffective2} Consider for $n \geq 2$ the Cauchy problem \eqref{general} with $m\in \big(\frac{\sqrt{n^2-1}}{2},\frac{n}{2}\big)$ and
 data $f \in H^\gamma(\R^n), \gamma \in (\frac{1}{2},1),$ and $g \in L^2(\R^n)$. Let $p>p_{n,\gamma}:=1+  \frac{2\gamma}{n-1}$ and $p \leq \frac{n}{n-2\gamma}$. 
  Then, there exists a constant $\varepsilon_0>0$ such that, for every small data satisfying
  \[\|f\|_{H^\gamma}+ \|g\|_{L^2}\leq \varepsilon\,\,\,\mbox{for}\,\,\, \varepsilon\leq \varepsilon_0,\]
  there exists a uniquely determined global (in time) Sobolev solution \[ \phi \in C\big([0,\infty),H^\gamma(\R^n)\big).\] The solution satisfies the decay estimate  \begin{equation*}
  \| \phi(t,\cdot)\|_{H^\gamma} \lesssim
  e^{\frac{-(n-1)t}{2}}\big(\|f\|_{H^\gamma}+ \|g\|_{L^2} \big).
 \end{equation*}
  \end{theorem}
\begin{proof}
We only sketch the proof, in particular, the modifications to the proof of Theorem \ref{noneffective}. We define the  space
\[ X(t) := \bigl\{ u\in \mathcal{C}\big([0,t], H^\gamma(\R^n) \big) \, : \ \|u\|_{X(t)} :=\sup_{\tau\in[0,t]} \big\{e^{\frac{(\mu-1)\tau}{2}}\|u(\tau,\cdot)\|_{H^\gamma}\big\}<\infty \bigr\}. \]
Using Propositions \ref{proposition1} and \ref{proposition11} for $s=0$ we have
\begin{equation}\label{eq:Kdatafractionalnoneffective}
\| K_0(t,0, x) \ast_{(x)} u_0(x) + K_1(t,0, x) \ast_{(x)} u_1(x) \|_{X(t)} \lesssim\,\|u_0\|_{H^\gamma} + \|u_1\|_{L^2}.
\end{equation}
Now fractional Gagliardo-Nirenberg inequality  comes into play.
We may estimate
\[\| |u(s,\cdot)|^{p}\|_{L^2}=\|u(s,\cdot)\|_{L^{2p}}^p\lesssim
\| u(s,\cdot)\|_{L^2}^{p(1-\theta)} \|u(s,\cdot)\|_{\dot{H}^\gamma}^{p\theta} \lesssim e^{\frac{(1-\mu)ps}{2}}\|u\|_{X(s)}^p,\]
where
\[ \theta= \frac{n}{\gamma}\Big(\frac12-\frac1{2p}\Big) \in [0,1], \quad p \leq
\frac{n}{n-2\gamma}.\]
Hence,
\[e^{\frac{(\mu-1)t}{2}}\|Gu(t,\cdot)\|_{L^2} \lesssim \|u\|_{X(t)}^p\int_0^t e^{(p-1)\left(r+\frac{1-\mu}{2}\right)s} ds\lesssim
 \|u\|_{X(t)}^p,\]
thanks to $r+\frac{1-\mu}{2}<0$ and $p>1$.
Propositions \ref{fractionalGagliardoNirenberg} and \ref{proposition11} imply for all $h \in L^2(\mathbb{R}^n)$ the estimate
\[\| K_1(t,s, x) \ast_{(x)} h\|_{\dot{H}^\gamma}\lesssim
\|  K_1(t,s, x) \ast_{(x)} h\|_{L^2}^{1-\gamma} \|K_1(t,s, x) \ast_{(x)} h\|_{\dot{H}^1}^{\gamma} \lesssim e^{\frac{(1+\mu)\gamma s}{2}}
e^{\frac{(1-\mu)\gamma t}{2}}\|h\|_{L^2}.\]
Finally, by using $p_{n,\gamma}<p\leq \frac{n}{n-2\gamma} $ and $r=\frac{-n + \mu}{2}$ we may estimate
\begin{eqnarray*}
 e^{\frac{(\mu-1) t}{2}}\|Gu(t,\cdot)\|_{\dot{H}^\gamma} &\lesssim & e^{\frac{(\mu-1)(1-\gamma) t}{2}}\|u\|_{X(t)}^p \int_0^t e^{(p-1)rs +\frac{(1-\mu)ps}{2}+\frac{(1+\mu)\gamma s}{2}} ds \\
& \lesssim &\int_0^t e^{(p-1)\left(r+ \frac{(1-\mu)}{2}\right)s +\gamma s} ds \lesssim\|u\|_{X(t)}^p.
 \end{eqnarray*}
Using H\"older's inequality and fractional Gagliardo-Nirenberg inequality,  we can follow the steps of the proof of the Lipschitz property in the proof of Theorem \ref{noneffective} to obtain
     \begin{eqnarray*}
\|Pu-Pv\|_{X(t)}
     \lesssim \|u-v\|_{X(t)} \bigl(\|u\|_{X(t)}^{p-1}+\|v\|_{X(t)}^{p-1}\bigr)
\end{eqnarray*}
for any~$u,v\in X(t)$.
This completes the proof.
\end{proof}
Finally, we are interested in energy solutions having a suitable higher regularity.
\begin{theorem}\label{noneffective3}
 Consider the Cauchy problem  \eqref{general} with $m\in \big(\frac{\sqrt{n^2-1}}{2},\frac{n}{2}\big)$ and data $f \in H^\sigma(\R^n)$ and $g \in H^{\sigma-1}(\R^n)$ for $n \geq 3$, where $\sigma \in \big(1,\frac{n}{2}\big)$.
Assume that $p$ satisfies the following condition:
\begin{eqnarray*}  \max\Big\{\frac{n+1}{n-1};\lceil \sigma
 \rceil\Big\} <p \leq 1+\frac{2}{n-2\sigma}.
 \end{eqnarray*}  Then, there exists a constant $\varepsilon_0>0$ such that, for every given small data satisfying
  \[\|f\|_{H^\sigma}+ \|g\|_{H^{\sigma-1}}\leq \varepsilon\,\,\,\mbox{for}\,\,\, \varepsilon\leq \varepsilon_0,\]
 there exists a uniquely determined global (in time) energy solution \[ \phi \in C\big([0,\infty),H^\sigma(\R^n)\big) \cap C^1\big([0,\infty),H^{\sigma-1}(\R^n)\big).\] Moreover, the solution $\phi$ satisfies the estimates \eqref{optidamp3energy} and \eqref{optidamp4energy}
  for $\gamma=\sigma$.
    \end{theorem}
   \begin{proof}
It is enough to prove the global existence of small data solutions to \eqref{semilieardampedwavedissipationtreatment}.
Motivated by the estimates of Proposition \ref{proposition1} and Proposition \ref{proposition11} for $s=0$ we define for $t>0$ the scale of spaces of energy solutions with suitable regularity
\begin{eqnarray*} && X(t) := \bigl\{ u\in C\big([0,t], H^\sigma(\R^n) ) \cap C^1([0,t],H^{\sigma-1}(\R^n)\big) \, : \, \|u\|_{X(t)} \\ && \quad :=\sup_{\tau\in[0,t]}\bigl\{e^{\frac{(\mu-1)\tau}{2}}\big(\|u(\tau,\cdot)\|_{L^2} + \||D|^\sigma u(\tau,\cdot)\|_{L^2}\big) +
e^{\mu\tau}\big(\|u_\tau(\tau,\cdot)\|_{L^2} +\||D|^{\sigma-1}u_\tau(\tau,\cdot)\|_{L^2}\big)\bigr\}
<\infty \bigr\}. \end{eqnarray*}
For any $u\in X(t)$ we define
\[ Pu :=  K_0(t,0, x) \ast_{(x)} u_0(x) + K_1(t,0, x) \ast_{(x)} u_1(x)  + Gu, \]
where
\[Gu(t,x):=\int_0^t e^{(p-1) r s}K_1(t,s,x) \ast_{(x)}|u(s,x)|^{p}\,ds.\]
 Following the proof of Theorem \ref{noneffective} and the universal treatment of
 non-linear terms in scales of Sobolev spaces done in  the proof
of Theorem \ref{main5}, one may derive that $Pu\in X(t)$ and the Lipschitz property
     \begin{eqnarray*}
\|Pu-Pv\|_{X(t)}
     \lesssim \|u-v\|_{X(t)} \bigl(\|u\|_{X(t)}^{p-1}+\|v\|_{X(t)}^{p-1}\bigr)
\end{eqnarray*}
for any $u,v\in X(t)$.
This completes the proof.
\end{proof}
  \begin{corollary} \label{Corlargeregularity11}
Consider the Cauchy problem  \eqref{general} with $m\in \big(\frac{\sqrt{n^2-1}}{2},\frac{n}{2}\big)$ and data $f \in H^{\sigma}(\R^n)$ and $g \in H^{\sigma-1}(\R^n)$ for $n \geq 3$ and
$n \leq 2\sigma$.
Assume that $p$ satisfies the following condition:
\begin{eqnarray*}  \max\Big\{\frac{n+1}{n-1};\lceil \sigma
 \rceil\Big\} <p<\infty.
 \end{eqnarray*}  Then, there exists a constant $\varepsilon_0>0$ such that, for every given small data satisfying
  \[\|f\|_{H^\sigma}+ \|g\|_{H^{\sigma-1}}\leq \varepsilon\,\,\,\mbox{for}\,\,\, \varepsilon\leq \varepsilon_0,\]
 there exists a uniquely determined global (in time) energy solution \[ \phi \in C\big([0,\infty),H^\sigma(\R^n)\big) \cap C^1\big([0,\infty),H^{\sigma-1}(\R^n)\big).\] Moreover, the solution $\phi$ satisfies the estimates \eqref{optidamp3energy} and \eqref{optidamp4energy}
  for $\gamma=\sigma$.
  \end{corollary}
  Similarly to Theorem \ref{improve1} we may improve the lower bound for $p$ in Corollary \ref{Corlargeregularity11}.
   \begin{theorem}\label{improve6} Consider the Cauchy problem  \eqref{general} with $m\in \big(\frac{\sqrt{n^2-1}}{2},\frac{n}{2}\big)$ and data $(f,g) \in (H^\sigma(\R^n)\times H^{\sigma-1}(\R^n))$ for
 $\sigma > \frac{n}{2} $ and $n \geq 2$.
Assume that $p$ satisfies the following condition:
\begin{eqnarray*}
\max\Big\{\frac{n+1}{n-1}; \sigma; 2
 \Big\} <p <\infty.
 \end{eqnarray*}
  Then, there exists a constant $\varepsilon_0>0$ such that, for every given small data satisfying
  \[\|f\|_{H^\sigma}+ \|g\|_{ H^{\sigma-1}} \leq \varepsilon\,\,\,\mbox{for}\,\,\, \varepsilon\leq \varepsilon_0,\]
 there exists a uniquely determined global (in time) energy solution \[ \phi \in C\big([0,\infty),H^\sigma(\R^n)\big)\cap C^1\big([0,\infty), H^{\sigma-1}(\R^n)\big).\]  Moreover, the solution $\phi$ satisfies the estimates \eqref{optidamp3energy} and \eqref{optidamp4energy}
  for $\gamma=\sigma$.
 \end{theorem}
  If  $\sigma +1< \frac{n+1}{n-1}$, with $\sigma > \frac{n}{2} $, that is,  $1 < \sigma < 2$ for $n=2$, one can also have a similar result like
  Theorem \ref{main1}. By using the embedding of $H^1(\R)$ into $L^{\infty}(\R)$ it is now  allowed to consider space dimension $n=1$, too.
  \begin{theorem}\label{onedimension} Consider the Cauchy problem  \eqref{general} with $m\in \big(\frac{\sqrt{n^2-1}}{2},\frac{n}{2}\big)$ and data $(f,g) \in (H^\sigma(\R^n)\times H^{\sigma-1}(\R^n))$ with $\sigma=1$ for $n=1$ and
 $1<\sigma <  2 $ for $n=2$. Let $p>\sigma +1$.
  Then, there exists a constant $\varepsilon_0>0$ such that, for every given small data satisfying
  \[\|f\|_{H^\sigma}+ \|g\|_{ H^{\sigma-1}} \leq \varepsilon\,\,\,\mbox{for}\,\,\, \varepsilon\leq \varepsilon_0,\]
 there exists a uniquely determined global (in time) energy solution \[ \phi \in C\big([0,\infty),H^\sigma(\R^n)\big)\cap C^1\big([0,\infty), H^{\sigma-1}(\R^n)\big).\]  Moreover, the solution $\phi$ satisfies the estimates \eqref{optidamp3energy} and \eqref{optidamp4energy}
  for $\gamma=\sigma$.
 \end{theorem}
 \begin{remark}
Following the proof of Theorem \ref{main1} and applying Proposition \ref{PropSickelfractional} for $s=1$, in the case $n=1$ one may weaken the condition on $p$ to $p-1> 1- \frac{1}{2}$, that is,
$p> \frac{3}{2}$.
 \end{remark}

\subsection{Concluding remarks} \label{Secconcludingdominantdissipationcase}
In Section \ref{Sec3.2} we proved several results for the global existence of small data solutions to the Cauchy problem
\begin{eqnarray*}
\phi_{tt} - e^{-2t} \Delta \phi + n\phi_t+m^2\phi=|\phi|^p,\,\,\,(\phi(0,x),\phi_t(0,x))=(f(x),g(x)),
\end{eqnarray*}
where $p>1$ and $m>0$. Let us choose $m=\frac{\sqrt{n^2-1}}{2}$. This implies $\mu=1$ in all the results.
Summarizing all the results allows the following conclusions:
\begin{enumerate}
\item If $n=2$, then for all $p \in (2,\infty)$ we can choose data $(f,g)\in H^\sigma \times H^{\sigma-1}$, $\sigma=\sigma(p)$, such that we have the global existence of small data solutions.
\item If $n=3$, then for all $p \in \big(\frac{3}{2},\infty\big)$ we can choose data $(f,g)\in H^\sigma \times H^{\sigma-1}$, $\sigma=\sigma(p)$, such that we have the global existence of small data solutions.
\item If $n=4$, then for all $p \in \big(\frac{4}{3},\infty\big)$ we can choose data $(f,g)\in H^\sigma \times H^{\sigma-1}$, $\sigma=\sigma(p)$, such that we have the global existence of small data solutions.
\item If $n=5$, then for all $p \in \big\{\big(\frac{5}{4},\frac{5}{3}\big] \cup (2,\infty)\big\}$ we can choose data $(f,g)\in H^\sigma \times H^{\sigma-1}$, $\sigma=\sigma(p)$, such that we have the global existence of small data solutions. We see, that there is a gap. We have no any result for $p \in \big(\frac{5}{3},2\big]$.
\item If $n=2m$, $m \geq 3$, then for all $p \in \big\{\big(\frac{2m}{2m-1},\frac{m}{m-1}\big] \cup (m,\infty)\big\}$ we can choose data $(f,g)\in H^\sigma \times H^{\sigma-1}$, $\sigma=\sigma(p)$, such that we have the global existence of small data solutions. A gap still exists.
    The statement of Theorem \ref{main5} can be applied for $\sigma \in \big(m-1+ \frac{m-2}{m-1}, m\big)$.
\item If $n=2m+1$, $m \geq 3$, then for all $p \in \big\{\big(\frac{2m+1}{2m},\frac{2m+1}{2m-1}\big] \cup (m+1,\infty)\big\}$ we can choose data $(f,g)\in H^\sigma \times H^{\sigma-1}$, $\sigma=\sigma(p)$, such that we have the global existence of small data solutions. A gap still exists.
    The statement of Theorem \ref{main5} can be applied for $\sigma \in \big(\frac{2m+1}{2}- \frac{1}{m}, \frac{2m+1}{2}\big)$.
\end{enumerate}

\section{De Sitter model with dominant mass: Case $m \in (\frac{n}{2},\infty)$. }\label{Sec4}
In this section we consider the Cauchy problem
\begin{equation}\label{semiliearwavemasstreatment}
\begin{cases}
u_{tt} - e^{-2t}\Delta u + \mu^2 \, u =e^{-\frac{n}{2} (p-1)t}|u|^p,
  \quad (t,x)\in (0,\infty)\times \R^n,
\\
u(0,x)= u_0(x), \quad u_t(0,x)=u_1(x),
  \quad x\in \R^n
\end{cases}
\end{equation}
with $\mu^2=m^2 - \frac{n^2}{4} >0$, that is, $m \in (\frac{n}{2},\infty)$.\\
According to Duhamel's principle, a solution of \eqref{semiliearwavemasstreatment} satisfies the non-linear integral equation
\[  u(t,x)=K_0(t,0, x) \ast_{(x)} u_0(x) + K_1(t,0, x) \ast_{(x)} u_1(x) + \int_0^t e^{-\frac{n}{2}(p-1)s}K_1(t,s,x) \ast_{(x)}|u(s,x)|^{p}\,ds, \]
where $K_j(t,0, x) \ast_{(x)} u_j(x)$, $j=0,1$, are the solutions to the corresponding linear Cauchy
problem
\begin{equation}\label{linearKleinGordon}
\begin{cases}
u_{tt} - e^{-2t}\Delta u + \mu^2 \,u=0,
  \quad (t,x)\in (0,\infty)\times \R^n,
\\
u(0,x)=\delta_{0j}u_0(x), \quad u_t(0,x)=\delta_{1j}u_1(x),
  \quad x\in \R^n,
\end{cases}
\end{equation}
with $\delta_{kj}=1$  for $k=j$, and zero otherwise. The term $K_1(t,s,x) \ast_{(x)}f(s,x)$ is the  solution  of the parameter-dependent Cauchy problem
\[\begin{cases}
u_{tt} - e^{-2t}\Delta u + \mu^2 u=0,
  \quad (t,x)\in (s,\infty)\times \R^n,
\\
u(s,x)=0, \quad u_t(s,x)=f(s,x),
  \quad x\in \R^n.
\end{cases}
\]
 So, Duhamel's principle explains that we have to take account of solutions
to a family of parameter-dependent Cauchy problems.
\subsection{Estimates of solutions to the corresponding linear model} \label{Sec4.1}
Let us consider for $t \geq s$ the parameter-dependent Cauchy problem for the Klein-Gordon type equation
\begin{equation}\label{massKleinGordonwavelinear}
\begin{cases}
u_{tt} - e^{-2t}\Delta u + \mu^2 u=0,
  \quad (t,x)\in (s,\infty)\times \R^n, \,\mu^2 >0,
\\
u(s,x)= \varphi(s,x), \quad u_t(s,x)=\psi(s,x),
  \quad x\in \R^n.
\end{cases}
\end{equation}
After application of partial Fourier transformation we have
\[ \hat{u}_{tt} + |\xi|^2e^{-2t} \hat{u} + \mu^2 \hat{u}=0.\]
If we introduce the change of variables  $\tau:=|\xi|e^{-t},\,\tau_0:=|\xi|e^{-s}$ and $v(\tau):=\hat{u}(t,\xi)$, then we get the ordinary differential equation
\[ v_{\tau \tau} + \frac{1}{\tau} v_\tau + \Big(1+\frac{\mu^2}{\tau^2}\Big)v=0.\]
If we define  $v(\tau)=\tau^\rho \tilde{v}(\tau)$, then after choosing $\rho=i\mu$ we arrive at
\[ \tau \tilde{v}_{\tau \tau} + (2\rho+1) \tilde{v}_{\tau} + \tau\tilde{v}=0.\]
Finally, the last equation is reduced to a confluent hypergeometric equation if we perform the change of variables $z:=2i \tau$ and $w(z)=e^{i\tau} \tilde{v}(\tau)$. In this way we obtain
\begin{equation}\label{hypergeometricEq2}
 z w_{zz} + (1+2\rho -z)w_z -\frac{1+2\rho}{2} w=0,
 \end{equation}
with the following initial conditions at $z_0=z_0(s,\xi)=2i\xii e^{-s}$:
\[w(z_0)=e^{i\xii e^{-s}}\frac{\widehat\varphi(s,\xi)}{(\xii e^{-s})^{\rho}}, \quad w'(z_0)=\frac{e^{i\xii e^{-s}}}{2}\Big(\frac{\widehat\varphi(s,\xi)}{(\xii e^{-s})^{\rho}}+i\frac{\rho\widehat\varphi(s,\xi)+\widehat\psi(s,\xi)}{(|\xi|e^{-s})^{\rho+1}}\Big).\]
     Due to \cite{BE} the general solution of \eqref{hypergeometricEq2} has the representation
\[w(z)=c_1(s,\xi)w_1(z)+ c_2(s,\xi)w_2(z),\]
where $w_1$ and $w_2$ are
 two linear independent solutions  given by
 \[w_1(z)=\Phi\Big(\frac{1+2i\mu}{2}, 1+2i\mu, z\Big), \quad w_2(z)=z^{-2i\mu}\Phi\Big(\frac{1-2i\mu}{2}, 1-2i\mu, z\Big). \]
Here $\Phi$ is the confluent hypergeometric function. Moreover, we can write
 \begin{equation}\label{coef2}
 c_j(s,\xi)=(-1)^{3-j}\frac{w(z_0)(d_z w_{3-j})(z_0)-(d_z w)(z_0)w_{3-j}(z_0)}{W(w_1,w_2)(z_0)} \,\,\,\mbox{for} \,\,\, j=1,2,
 \end{equation}
 where $W(w_1,w_2)$ is the Wronskian of the two linear independent solutions. The Wronskian (\cite{BE},vol.1,p.253) is equal to
\[W(w_1,w_2)(z)=w_1\frac{d}{dz}w_2-w_2\frac{d}{dz}w_1=-2i\mu z^{-2i\mu-1}e^z.\]
Since $z$ is a pure imaginary number, it follows $|W(w_1,w_2)(z)|\approx |z|^{-1}$.
Using all these representations we conclude the following WKB representation
for $\widehat{u}$:
\begin{eqnarray*}
  && \widehat{u}(t,\xi)= e^{i\mu\ln(\xii e^{-t})-i\xii e^{-t}}\Big(c_1(s,\xi)\Phi\Big(\frac{1+2i\mu}{2}, 1+2i\mu, 2i\xii e^{-t}\Big)\\
  && \qquad +c_2(s,\xi) (2i\xii e^{-t})^{-2i\mu}
  \Phi\Big(\frac{1-2i\mu}{2}, 1-2i\mu, 2i\xii e^{-t} \Big) \Big),
   \end{eqnarray*}
  where the coefficients $c_j(s,\xi)$ for $j=1,2$ are given by \eqref{coef2} with $z_0=2i\xii e^{-s}$.\\
Now we may follow the approach of the previous section. For this reason we split the extended phase space into three zones
\[ Z_1(s)
     := \{\xi: \ \xii e^{-s}\leq N \},\ \ Z_2(t,s)
     := \{N e^{s}\leq \xii \leq N e^{t} \},  \ \ Z_3(t)
     := \{\xii e^{-t}\geq N \}.\]
In this way we are able to prove the following results.
\begin{proposition}\label{propositionKleinGordon1}
Assume that $\varphi(0,\cdot) \in H^{\gamma}(\R^n)$, $\gamma \geq 0$ and $\psi(0,\cdot) \equiv 0$. Then the following estimates hold for the solutions to (\ref{massKleinGordonwavelinear}) for $t \in [0,\infty)$:
\begin{eqnarray*}
&& \| K_0(t,0, x) \ast_{(x)} \varphi(0,x)\|_{\dot{H}^{\gamma}} \lesssim
  e^{\frac{t}{2}}\|\varphi(0,\cdot)\|_{\dot{H}^{\gamma}},\\
&&  \|\partial_t K_0(t,0, x) \ast_{(x)} \varphi(0,x)\|_{\dot{H}^{\gamma-1}} \lesssim
  \|\varphi(0,\cdot)\|_{\dot{H}^{\gamma}}\,\,\,\mbox{for}\,\,\, \gamma \geq 1.
\end{eqnarray*}
\end{proposition}
\begin{proposition}\label{propositionKleinGordon2parameter}
Assume that $\psi(s,\cdot) \in \dot{H}^\gamma(\R^n),\,\gamma \geq 0,$ for $s \in [0,\infty)$ and $\varphi(s,\cdot) \equiv 0$. Then the following estimates hold for the solutions to (\ref{massKleinGordonwavelinear}) for $t \in [s,\infty)$:
\begin{eqnarray*}
&& \|K_1(t,s, x) \ast_{(x)} \psi(s,x)\|_{\dot{H}^\gamma} \lesssim
  \|\psi(s,\cdot)\|_{\dot{H}^\gamma},\\
&& \|\nabla K_1(t,s, x) \ast_{(x)} \psi(s,x)\|_{\dot{H}^\gamma} \lesssim
  e^{\frac{s+t}{2}}  \|\psi(s,\cdot)\|_{\dot{H}^\gamma}, \\
&& \|\partial_t K_1(t,s, x) \ast_{(x)} \psi(s,x)\|_{\dot{H}^\gamma} \lesssim
  \|\psi(s,\cdot)\|_{\dot{H}^\gamma}.
\end{eqnarray*}
\end{proposition}
\begin{proof}
We only sketch the proof of the last two propositions.
\begin{enumerate}
\item In $Z_1$, by using properties (P1) and (P2) of Proposition \ref{hypergeometricFucntions} we have
\[ |c_1(s, \xi)|\lesssim \xii e^{-s}\Big(\frac{|\widehat\varphi(s,\xi)|}{\xii e^{-s}}+|\widehat\varphi(s,\xi)| +\frac{|\widehat\psi(s,\xi)|}{\xii e^{-s}}\Big)
\lesssim |\widehat\varphi(s,\xi)|+|\widehat\psi(s,\xi)|,  \]
and
\[ |c_2(s,\xi)|\lesssim \xii e^{-s}\Big(|\widehat\varphi(s,\xi)|+ \frac{  |\widehat\varphi(s,\xi)|+|\widehat\psi(s,\xi)|}{\xii e^{-s}}\Big)
\lesssim |\widehat\varphi(s,\xi)|+|\widehat\psi(s,\xi)|.\]
So, for $\gamma\geq 0$ we can estimate
\[\xii^{\gamma}|\widehat{u}(t,\xi)|\lesssim \xii^{\gamma} (|\widehat\varphi(s,\xi)|+|\widehat\psi(s,\xi)|).\]
If $\psi\in \dot{H}^{\gamma-1}$ but not in $\dot{H}^\gamma$, then we may only conclude
\[\xii^\gamma\widehat{u}(t,\xi)|\lesssim  \xii^\gamma|\widehat\varphi(s,\xi)|+e^{s}\xii^{\gamma-1}|\widehat\psi(s,\xi)|. \]
Using $|\partial_t e^{i\mu\ln(\xii e^{-t})-i\xii e^{-t}}|\leq C$ in $Z_1$ it follows immediately
\[\xii^\gamma|\widehat{u}_t(t,\xi)|\lesssim \xii^\gamma \big(|\widehat\varphi(s,\xi)|+|\widehat\psi(s,\xi)|\big).\]
\item In $Z_3$ we have that $\xii e^{-s}\geq\xii e^{-t}\geq N$ and thanks to Proposition \ref{hypergeometricFucntions} we can estimate
\begin{equation}\label{coefdamp2}
  |c_j(s, \xi)|\lesssim \xii e^{-s}\Big(\!|\widehat\varphi(s,\xi)|+\frac{|\widehat\varphi(s,\xi)|+|\widehat\psi(s,\xi)|}{\xii e^{-s}}\!\Big)
  \!(\xii e^{-s})^{-\frac{1}{2}}\lesssim\!|\widehat\varphi(s,\xi)|(\xii e^{-s})^{\frac{1}{2}}+|\widehat\psi(s,\xi)|(\xii e^{-s})^{-\frac{1}{2}}.
 \end{equation}
So,  by using \eqref{coefdamp2} and property (P3) of Proposition \ref{hypergeometricFucntions} we may conclude
   \begin{eqnarray*}
  && \xii^{\gamma}|\widehat{u}(t,\xi)| \lesssim  \xii^{\gamma} \Big(|\widehat\varphi(s,\xi)|(\xii e^{-s})^{\frac{1}{2}}+|\widehat\psi(s,\xi)|(\xii e^{-s})^{-\frac{1}{2}}\Big)
  (\xii e^{-t})^{-\frac{1}{2}}\\
  && \qquad \lesssim
   \xii^{\gamma}\Big(|\widehat\varphi(s,\xi)|e^{-\frac{s}{2}}+ \frac{|\widehat\psi(s,\xi)|}{\xii}e^{\frac{s}{2}}\Big)
   e^{\frac{t}{2}}
  \lesssim
   \xii^{\gamma}\big(|\widehat\varphi(s,\xi)|e^{-\frac{s-t}{2}}+ |\widehat\psi(s,\xi)|e^{\frac{s-t}{2}}\big).
  \end{eqnarray*}
In order to avoid an exponential increasing term in $t$ we use regularity in the last inequality, i.e., we use the estimate $N\xii^{-1}\leq e^{-t}$. If $\psi(s,\cdot)\in \dot{H}^{\gamma-1}$ but not in $\dot{H}^\gamma$, then one may only derive
 \[ \xii^\gamma|\widehat{u}(t,\xi)| \lesssim
   \xii^\gamma |\widehat\varphi(s,\xi)|e^{-\frac{s-t}{2}}+ \xii^{\gamma-1}|\widehat\psi(s,\xi)|e^{\frac{s+t}{2}}.\]
Similarly, using that $|\partial_t e^{i\mu\ln(\xii e^{-t})-i\xii e^{-t}}|\leq C\xii e^{-t}$ in $Z_3$, we may conclude
   \[\xii^{\gamma-1}|\widehat{u}_t(t,\xi)|\lesssim \xii^\gamma e^{-t} \Big(
   |\widehat\varphi(s,\xi)|e^{-\frac{s-t}{2}}+ \frac{|\widehat\psi(s,\xi)|}{\xii}e^{\frac{s+t}{2}}\Big)\lesssim \xii^\gamma |\widehat\varphi(s,\xi)|e^{-\frac{s+t}{2}}+ \xii^{\gamma-1}|\widehat\psi(s,\xi)|e^{\frac{s-t}{2}}.\]
 \item In $Z_2$ we still use \eqref{coefdamp2} to conclude
\begin{eqnarray*}
  && \xii^{\gamma}|\widehat{u}(t,\xi)| \lesssim  \xii^{\gamma}\Big(|\widehat\varphi(s,\xi)|(\xii e^{-s})^{\frac{1}{2}}+|\widehat\psi(s,\xi)|(\xii e^{-s})^{-\frac{1}{2}}\Big)\\
  && \qquad \lesssim \xii^{\gamma}\big(|\widehat\varphi(s,\xi)|e^{-\frac{s-t}{2}}+ |\widehat\psi(s,\xi)|\big).
    \end{eqnarray*}
If $\psi\in \dot{H}^{\gamma-1}$, but not in $\dot{H}^\gamma$, one may only derive
      \[ \xii^\gamma|\widehat{u}(t,\xi)| \lesssim
   \xii^\gamma |\widehat\varphi(s,\xi)|e^{-\frac{s-t}{2}}+ \xii^{\gamma-1}|\widehat\psi(s,\xi)|e^{\frac{s+t}{2}}.\]
     Similarly, we conclude
\[\xii^{\gamma-1}|\widehat{u}_t(t,\xi)|\lesssim \xii^\gamma e^{-s} |\widehat\varphi(s,\xi)|+ \xii^{\gamma-1}|\widehat\psi(s,\xi)|.\]
\end{enumerate}
This completes the proof. \end{proof}

\subsection{Global existence of small data solutions} \label{Sec4.2}
Firstly we are interested in the global existence (in time) of energy solutions.
\begin{theorem}\label{main3} Consider for $n \geq 2$ the Cauchy problem  \eqref{general} with $m\in (\frac{n}{2},\infty)$  and
data $(f,g) \in (H^1(\R^n)\times L^2(\R^n))$. Let $p>\frac{n+1}{n-1}$  and $p\leq \frac{n}{n-2}$ for $n\geq 3$.
  Then, there exists a constant $\varepsilon_0>0$ such that, for every small data satisfying
  \[\|f\|_{H^1}+ \|g\|_{L^2}\leq \varepsilon\,\,\,\mbox{for}\,\,\, \varepsilon\leq \varepsilon_0,\]
  there exists a uniquely determined global (in time) energy solution \[ \phi \in C\big([0,\infty),H^1(\R^n)) \cap C^1([0,\infty),L^2(\R^n)\big).\]
  The energy solution satisfies the decay estimate
  \begin{eqnarray*}
  \|\phi(t,\cdot)\|_{L^2} + \|\nabla \phi(t,\cdot)\|_{L^2}+ \|\phi_t(t,\cdot)\|_{L^2}  \lesssim e^{-\frac{n-1}{2}t}
  \big(\|f\|_{H^1} + \|g\|_{L^2} \big).
  \end{eqnarray*}
  \end{theorem}
\begin{proof}
It is enough to prove the global existence of small data solutions to \eqref{semiliearwavemasstreatment}.
Motivated by the estimates of Proposition \ref{propositionKleinGordon1} and Proposition \ref{propositionKleinGordon2parameter} for $s=0$ we define for $t>0$ the scale of spaces of energy solutions
\begin{eqnarray*} && X(t) := \bigl\{ u\in C\big([0,t], H^1(\R^n) ) \cap C^1([0,t],L^2(\R^n)\big)\\ && \quad : \, \|u\|_{X(t)} :=\sup_{\tau\in[0,t]}\bigl\{e^{-\frac{\tau}{2}}\|u(\tau,\cdot)\|_{L^2} + e^{-\frac{\tau}{2}}\|\nabla u(\tau,\cdot)\|_{L^2} +\|u_\tau(\tau,\cdot)\|_{L^2} \bigr\}
<\infty \bigr\}. \end{eqnarray*}
For any $u\in X(t)$ we define
\[ Pu :=  K_0(t,0, x) \ast_{(x)} u_0(x) + K_1(t,0, x) \ast_{(x)} u_1(x)  + Gu, \]
where
\[Gu(t,x):=\int_0^t e^{-\frac{n}{2}(p-1)s}K_1(t,s,x) \ast_{(x)}|u(s,x)|^{p}\,ds.\]
Using Proposition \ref{propositionKleinGordon1} and Proposition \ref{propositionKleinGordon2parameter} for $s=0$ we have
\begin{equation}\label{eq:Kdatamass}
\| K_0(t,0, x) \ast_{(x)} u_0(x) + K_1(t,0, x) \ast_{(x)} u_1(x) \|_{X(t)} \lesssim\,\|u_0\|_{H^1} + \|u_1\|_{L^2}.
\end{equation}
 Using Minkowski's integral inequality gives with Proposition \ref{propositionKleinGordon2parameter}
 \[\|Gu(t,\cdot)\|_{L^2} \lesssim \int_0^t e^{-\frac{n}{2}(p-1)s}\|K_1(t,s,x) \ast_{(x)}|u(s,x)|^{p}\|_{L^2}\,ds\lesssim \int_0^t e^{-\frac{n}{2}(p-1)s}
 \| |u(s,x)|^{p}\|_{L^2}\,ds.\]
 Now Gagliardo-Nirenberg inequality  comes into play.
We may estimate
\[\| |u(s,\cdot)|^{p}\|_{L^2}=\|u(s,\cdot)\|_{L^{2p}}^p\lesssim
\| u(s,\cdot)\|_{L^2}^{p(1-\theta)} \| \nabla u(s,\cdot)\|_{L^2}^{p\theta} \lesssim e^{\frac{s}{2}p}\|u\|_{X(s)}^p,\]
where
\[ \theta= n\Big(\frac12-\frac1{2p}\Big), \qquad 2p \leq \begin{cases}
\infty & \text{if~$n\leq 2$,}\\
\frac{2n}{n-2} & \text{if~$n\geq 3$.}
\end{cases} \]
Hence,
\begin{eqnarray*}  && \|Gu(t,\cdot)\|_{L^2} \lesssim \|u\|_{X(t)}^p \int_0^t e^{-\frac{n}{2}(p-1)s+\frac{s}{2}p} ds \\
&& \qquad =\|u\|_{X(t)}^p \int_0^t e^{-\frac{n-1}{2}(p-1)s+\frac{s}{2}} ds
 \lesssim e^{\frac{t}{2}}\|u\|_{X(t)}^p,\end{eqnarray*}
thanks to $p>1$. In the same way we get after applying Propositions \ref{propositionKleinGordon1} and \ref{propositionKleinGordon2parameter} for
\[\partial_t Gu(t,x)=\int_0^t e^{-\frac{n}{2}(p-1)s} \partial_t (K_1(t,s,x) \ast_{(x)}|u(s,x)|^{p})\,ds\]
the estimate
\[\| \partial_t   Gu(t,\cdot)\|_{L^2} \lesssim \|u\|_{X(t)}^p \int_0^t e^{-\frac{n}{2}(p-1)s+\frac{s}{2}p} ds \lesssim \|u\|_{X(t)}^p \]
for all $p> \frac{n}{n-1}$. Finally, it remains to estimate
\[e^{-\frac{t}{2}}\nabla Gu(t,x)=e^{-\frac{t}{2}}\int_0^t e^{-\frac{n}{2}(p-1)s} \nabla (K_1(t,s,x) \ast_{(x)}|u(s,x)|^{p})\,ds.\]
Again the application of Propositions \ref{propositionKleinGordon1} and \ref{propositionKleinGordon2parameter} yields
\begin{eqnarray*}
&& \| e^{-\frac{t}{2}}\nabla Gu(t,\cdot)\|_{L^2} \lesssim e^{-\frac{t}{2}}\int_0^t e^{-\frac{n}{2}(p-1)s} \|\nabla (K_1(t,s,x) \ast_{(x)}|u(s,x)|^{p})\|_{L^2}\,ds \\
&& \qquad \lesssim e^{-\frac{t}{2}}\int_0^t e^{-\frac{n}{2}(p-1)s +  \frac{s}{2} p + \frac{t+s}{2}} \,ds \|u\|_{X(t)}^p
\\
&& \qquad \lesssim \int_0^t e^{-\frac{n}{2}(p-1)s +  \frac{s}{2} p + \frac{s}{2}} \,ds \|u\|_{X(t)}^p \lesssim \|u\|_{X(t)}^p
\end{eqnarray*}
thanks to the assumption $p > \frac{n+1}{n-1}$.
Summarizing all the derived estimates it follows for all $t>0$ and $p>\frac{n+1}{n-1}$ the estimate
\[ \|Pu\|_{X(t)}
     \lesssim\,\|u_0\|_{H^1} + \|u_1\|_{L^2}+ \|u\|_{X(t)}^{p}.
     \]
This leads to $Pu\in X(t)$. Following the steps to show the Lipschitz property from the proof to Theorem \ref{main2} implies
     \begin{equation}
\label{eq:contractionmass}
\|Pu-Pv\|_{X(t)}
     \lesssim\|u-v\|_{X(t)} \bigl(\|u\|_{X(t)}^{p-1}+\|v\|_{X(t)}^{p-1}\bigr)
\end{equation}
for any $u,v\in X(t)$. Using these estimates for $Pu$ one can prove the existence of a
uniquely determined global (in time) energy solution $u$ by contraction argument for small data. Moreover, we get a local (in time) result for large data. The decay estimates of Theorem \ref{main3} follow by using the relation $\phi(t,x)=e^{-\frac{n}{2}t}u(t,x)$.
\end{proof}
\begin{remark}
As in \cite{Y}, we conclude  that
$\|\nabla \phi(t,\cdot)\|_{H^1}$ is bounded for all $n\geq 2$.
Is it possible to allow a loss of decay for solutions or to change the data classes in order to have global existence for all $p>1$? Let us discuss a possible loss of decay. For this reason we define the space
\begin{eqnarray*} && X_0(t) := \bigl\{ u\in C\big([0,t], H^1(\R^n) ) \\ && \quad : \, \|u\|_{X_0(t)} :=\sup_{\tau\in[0,t]}\bigl\{e^{-\alpha \tau}\|u(\tau,\cdot)\|_{L^2} + e^{-\beta \tau}\|\nabla u(\tau,\cdot)\|_{L^2} \bigr\}
<\infty \bigr\} \end{eqnarray*}
with suitable real parameters $\alpha$ and $\beta$. The estimates of Propositions \ref{propositionKleinGordon1} and  \ref{propositionKleinGordon2parameter} require $\alpha \geq \frac{1}{2}$ and $\beta \geq \frac{1}{2}$. We are interested under which assumptions to $p,\alpha$ and $\beta$ the operator
\[ G: u \in X_0(t) \to Gu:=:=\int_0^t e^{-\frac{n}{2}(p-1)s} K_1(t,s,x) \ast_{(x)}|u(s,x)|^{p}\,ds \]
maps $X_0(t)$ into itself for all $t>0$. Following the estimates of the proof to Theorem \ref{main2} we obtain the following conditions:
\[ \frac{n}{2}(p-1)-\alpha p(1-\theta)-\beta p \theta +\alpha >0,\,\,\,\frac{n}{2}(p-1)-\alpha p(1-\theta)-\beta p \theta +\beta -1 >0.\]
Taking account of the definition of $\theta$ these conditions are equivalent to
\[ \Big(\frac{n}{2}+\frac{\alpha n}{2} -\frac{\beta n}{2} -\alpha \Big)(p-1) >0,\,\,\,\Big(\frac{n}{2}+\frac{\alpha n}{2} -\frac{\beta n}{2} -\alpha \Big)(p-1) >\alpha +1 -\beta.\]
Let $\beta \geq \alpha+1$, then the first condition cannot be satisfied for any $p>1$. So, let $\frac{1}{2} \leq \beta <\alpha +1$. We introduce $\beta=\alpha+1-z$ with $z \in (0,\frac{\alpha}{2}+1]$. We have to check the second condition only. We get
\[ \Big(\frac{nz}{2} - \alpha\Big)(p-1) >z.   \]
We have to choose $z \in \big(\frac{2\alpha}{n}, \alpha +\frac{1}{2}\big]$. Then we obtain the restriction $p> \frac{(n+2)z - 2\alpha}{nz-2\alpha}$. This bound is minimal for maximal $z=\alpha +\frac{1}{2}$. Hence, we conclude
$p> \frac{n(\alpha+\frac{1}{2})+1}{n(\alpha+\frac{1}{2})-2\alpha}$ for $\alpha \in [\frac{1}{2},\infty)$. The bound to below is strictly increasing in $\alpha$. Thus the minimal value is taken for $\alpha=\frac{1}{2}$ and, consequently, $\beta=\frac{1}{2}$, what we have chosen in Theorem \ref{main2}. Summarizing, a possible loss of decay does not lower the bound for $p$ to below.
\end{remark}
In the following result we will not require energy solutions any more. We restrict ourselves to Sobolev solutions only.
\begin{theorem}\label{main25} Consider for $n \geq 2$ the Cauchy problem \eqref{general} with $m\in (\frac{n}{2},\infty)$, data $f \in H^\gamma(\R^n), \gamma \in (\frac{1}{2},1)$, and $g \in L^2(\R^n)$. Let $p>p_{n,\gamma}:=1+\frac{2\gamma}{n-1}$ and $p \leq \frac{n}{n-2\gamma}$.
  Then, there exists a constant $\varepsilon_0>0$ such that, for every small data satisfying
  \[\|f\|_{H^\gamma}+ \|g\|_{L^2}\leq \varepsilon\,\,\,\mbox{for}\,\,\, \varepsilon\leq \varepsilon_0,\]
  there exists a uniquely determined global (in time) Sobolev solution \[ \phi \in C\big([0,\infty),H^\gamma(\R^n)\big).\]
  The solution satisfies the decay estimate  \begin{equation*}
  \| \phi(t,\cdot)\|_{H^\gamma} \lesssim
  e^{-\frac{n-1}{2}t}\big(\|f\|_{H^\gamma}+ \|g\|_{L^2} \big).
 \end{equation*}
  \end{theorem}
  \begin{remark}
  The assumption $p_{n,\gamma} < p \leq \frac{n}{n-2\gamma}$ implies $n\geq 2$ and   $\gamma > \frac{1}{2}$.
  \end{remark}
\begin{proof}
We only sketch the proof, in particular, the modifications to the proof of Theorem \ref{main3}.  It is enough to prove the global existence of small data solutions to \eqref{semiliearwavemasstreatment}.
Motivated by the estimates of Proposition \ref{propositionKleinGordon1} and Proposition \ref{propositionKleinGordon2parameter} for $s=0$ we define for $t>0$ the  space
\begin{eqnarray*} X(t) := \bigl\{ u\in C\big([0,t], H^{\gamma}(\R^n) )  \quad : \, \|u\|_{X(t)} :=\sup_{\tau\in[0,t]}\bigl\{e^{-\frac{\tau}{2}}\|u(\tau,\cdot)\|_{H^{\gamma}}  \bigr\}
<\infty \bigr\}. \end{eqnarray*}
For any $u\in X(t)$ we define
\[ Pu :=  K_0(t,0, x) \ast_{(x)} u_0(x) + K_1(t,0, x) \ast_{(x)} u_1(x)  + Gu, \]
where
\[Gu(t,x):=\int_0^t e^{-\frac{n}{2}(p-1)s}K_1(t,s,x) \ast_{(x)}|u(s,x)|^{p}\,ds.\]
Using Proposition \ref{propositionKleinGordon1} and Proposition \ref{propositionKleinGordon2parameter} for $s=0$ we have
\begin{equation*}
\| K_0(t,0, x) \ast_{(x)} u_0(x) + K_1(t,0, x) \ast_{(x)} u_1(x) \|_{X(t)} \lesssim\,\|u_0\|_{H^{\gamma}} + \|u_1\|_{L^2}.
\end{equation*}
 Using Minkowski's integral inequality gives with Proposition \ref{propositionKleinGordon2parameter}
 \begin{eqnarray*} && e^{-\frac{t}{2}}\|Gu(t,\cdot)\|_{L^2} \lesssim e^{-\frac{t}{2}}\int_0^t e^{-\frac{n}{2}(p-1)s}\|K_1(t,s,x) \ast_{(x)}|u(s,x)|^{p}\|_{L^2}\,ds\\ && \qquad \lesssim e^{-\frac{t}{2}}\int_0^t e^{-\frac{n}{2}(p-1)s}
 \| |u(s,x)|^{p}\|_{L^2}\,ds.\end{eqnarray*}
 Now fractional Gagliardo-Nirenberg inequality  comes into play.
We may estimate
\[\| |u(s,\cdot)|^{p}\|_{L^2}=\|u(s,\cdot)\|_{L^{2p}}^p\lesssim
\| u(s,\cdot)\|_{L^2}^{p(1-\theta)} \|  u(s,\cdot)\|_{\dot{H}^{\gamma}}^{p\theta} \lesssim e^{\frac{s}{2}p}\|u\|_{X(s)}^p,\]
where
\[ \theta= \frac{n}{\gamma}\Big(\frac12-\frac1{2p}\Big), \qquad p \leq
\frac{n}{n-2\gamma}.
 \]
Hence,
\begin{eqnarray*}  && e^{-\frac{t}{2}}\|Gu(t,\cdot)\|_{L^2} \lesssim \|u\|_{X(t)}^p e^{-\frac{t}{2}}\int_0^t e^{-\frac{n}{2}(p-1)s+\frac{s}{2}p} ds \\
&& \qquad =\|u\|_{X(t)}^p e^{-\frac{t}{2}}\int_0^t e^{-\frac{n-1}{2}(p-1)s+\frac{s}{2}} ds
 \lesssim \|u\|_{X(t)}^p,\end{eqnarray*}
thanks to $p>1$.
Propositions \ref{fractionalGagliardoNirenberg} and \ref{propositionKleinGordon2parameter} imply for all $h \in L^2(\mathbb{R}^n)$ the estimate
\[\| K_1(t,s,x) \ast_{(x)} h\|_{\dot{H}^{\gamma}} \leq \| K_1(t,s,x) \ast_{(x)} h\|_{\dot{H}^{1}}^{\gamma} \| K_1(t,s,x) \ast_{(x)} h\|_{L^2}^{1-\gamma}
\leq e^{\frac{(s+t)\gamma}{2}} \|  h\|_{L^2}. \]
 Together with Proposition \ref{propositionKleinGordon2parameter} we may conclude
\begin{eqnarray*}
&& e^{-\frac{t}{2}}\|  Gu(t,\cdot)\|_{\dot{H}^{\gamma}} \lesssim e^{-\frac{t}{2}}\int_0^t e^{-\frac{n}{2}(p-1)s} \| (K_1(t,s,x) \ast_{(x)}|u(s,x)|^{p})\|_{\dot{H}^{\gamma}}\,ds \\
&& \qquad \lesssim e^{-\frac{t}{2}}\int_0^t e^{-\frac{n}{2}(p-1)s +  \frac{s}{2} p + \frac{(t+s)\gamma}{2}} \,ds \|u\|_{X(t)}^p
 \lesssim \|u\|_{X(t)}^p
\end{eqnarray*}
thanks to the assumption $ 1+  \frac{2\gamma}{n-1}<p \leq\frac{n}{n-2\gamma}$.
Following the steps to show the Lipschitz property from the proofs to Theorems \ref{main2} and \ref{main20}
implies
     \begin{equation*}
\|Pu-Pv\|_{X(t)}
     \lesssim\|u-v\|_{X(t)} \bigl(\|u\|_{X(t)}^{p-1}+\|v\|_{X(t)}^{p-1}\bigr)
\end{equation*}
for any $u,v\in X(t)$. As before we conclude a global (in time) result of Sobolev solutions for small data and a local (in time)
result for large data as well.
By using the relation $\phi(t,x)=e^{-\frac{n}{2}\,t}u(t,x)$ we derive the decay estimate
\begin{eqnarray*}
  \|\phi(t,\cdot)\|_{{H}^{\gamma}}   \lesssim e^{-\frac{n-1}{2}t}
  \big(\|f\|_{{H}^{\gamma}} + \|g\|_{L^2} \big).
  \end{eqnarray*}
This completes the proof.
\end{proof}
Finally, we are interested in energy solutions having a suitable higher regularity.
\begin{theorem}\label{main8} Consider the Cauchy problem  \eqref{general} with $m\in (\frac{n}{2},\infty)$ and data $(f,g) \in (H^\sigma(\R^n)\times H^{\sigma-1}(\R^n))$ for $n \geq 3$, where $\sigma \in \big(1,\frac{n}{2}\big)$.
Assume that $p$ satisfies the following condition:
\begin{eqnarray*}  \max\Big\{\frac{n+1}{n-1};\lceil \sigma
 \rceil\Big\} <p \leq 1+\frac{2}{n-2\sigma}.
 \end{eqnarray*}
  Then, there exists a constant $\varepsilon_0>0$ such that, for every small data satisfying
  \[\|f\|_{H^\sigma}+ \|g\|_{H^{\sigma-1}}\leq \varepsilon\,\,\,\mbox{for}\,\,\, \varepsilon\leq \varepsilon_0,\]
  there exists a uniquely determined global (in time) energy solution \[ \phi \in C\big([0,\infty),H^\sigma(\R^n)) \cap C^1([0,\infty),H^{\sigma-1}(\R^n)\big).\]
  The energy solution satisfies the decay estimate
  \begin{eqnarray*}
  \|\phi(t,\cdot)\|_{H^\sigma} + \|\phi_t(t,\cdot)\|_{H^{\sigma-1}}  \lesssim e^{-\frac{n-1}{2}t}
  \big(\|f\|_{H^\sigma} + \|g\|_{H^{\sigma-1}} \big).
  \end{eqnarray*}
  \end{theorem}
\begin{proof}
We only sketch the proof.  It is enough to prove the global existence of small data solutions to \eqref{semiliearwavemasstreatment}.
Motivated by the estimates of Proposition \ref{propositionKleinGordon1} and Proposition \ref{propositionKleinGordon2parameter} for $s=0$ we define for $t>0$ the scale of spaces of energy solutions with suitable regularity
\begin{eqnarray*} && X(t) := \bigl\{ u\in C\big([0,t], H^\sigma(\R^n) ) \cap C^1([0,t],H^{\sigma-1}(\R^n)\big)\\ && \quad : \, \|u\|_{X(t)} :=\sup_{\tau\in[0,t]}\bigl\{e^{-\frac{\tau}{2}}\|u(\tau,\cdot)\|_{L^2} + e^{-\frac{\tau}{2}}\||D|^\sigma u(\tau,\cdot)\|_{L^2} +\|u_\tau(\tau,\cdot)\|_{L^2} +\||D|^{\sigma-1}u_\tau(\tau,\cdot)\|_{L^2}\bigr\}
<\infty \bigr\}. \end{eqnarray*}
For any~$u\in X(t)$ we define
\[ Pu :=  K_0(t,0, x) \ast_{(x)} u_0(x) + K_1(t,0, x) \ast_{(x)} u_1(x)  + Gu, \]
where
\[Gu(t,x):=\int_0^t e^{-\frac{n}{2}(p-1)s}K_1(t,s,x) \ast_{(x)}|u(s,x)|^{p}\,ds.\]
 Following the proof of Theorem \ref{main3} and the universal treatment of
 non-linear terms in scales of Sobolev spaces done in  the proof
of Theorem \ref{main5}, one may derive that $Pu\in X(t)$ and the Lipschitz property
     \begin{eqnarray*}
\|Pu-Pv\|_{X(t)}
     \lesssim \|u-v\|_{X(t)} \bigl(\|u\|_{X(t)}^{p-1}+\|v\|_{X(t)}^{p-1}\bigr)
\end{eqnarray*}
for any~$u,v\in X(t)$.
This completes the proof.
\end{proof}
If $n \leq 2\sigma$, then we can follow the steps of the previous proof and get the following statement.
\begin{corollary} \label{Corlargeregularity}
Consider the Cauchy problem  \eqref{general} with $m\in (\frac{n}{2},\infty)$ and data $f \in H^{\sigma}(\R^n)$ and $g \in H^{\sigma-1}(\R^n)$ for $n \geq 3$ and
$n \leq 2\sigma$.
Assume that $p$ satisfies the following condition:
\begin{eqnarray*}  \max\Big\{\frac{n+1}{n-1};\lceil \sigma
 \rceil\Big\} <p<\infty.
 \end{eqnarray*}
  Then, there exists a constant $\varepsilon_0>0$ such that, for every small data satisfying
  \[\|f\|_{H^\sigma}+ \|g\|_{H^{\sigma-1}}\leq \varepsilon\,\,\,\mbox{for}\,\,\, \varepsilon\leq \varepsilon_0,\]
  there exists a uniquely determined global (in time) energy solution \[ \phi \in C\big([0,\infty),H^\sigma(\R^n)) \cap C^1([0,\infty),H^{\sigma-1}(\R^n)\big).\]
  The energy solution satisfies the decay estimate
  \begin{eqnarray*}
  \|\phi(t,\cdot)\|_{H^\sigma} + \|\phi_t(t,\cdot)\|_{H^{\sigma-1}}  \lesssim e^{-\frac{n-1}{2}t}
  \big(\|f\|_{H^\sigma} + \|g\|_{H^{\sigma-1}} \big).
  \end{eqnarray*}
\end{corollary}
Similarly to Theorem \ref{improve1} we may improve the lower bound for $p$ in Corollary \ref{Corlargeregularity}.
   \begin{theorem}\label{improve7} Consider the Cauchy problem  \eqref{general} for $m\in (\frac{n}{2},\infty)$ and data $(f,g) \in (H^\sigma(\R^n)\times H^{\sigma-1}(\R^n))$ for
 $\sigma > \frac{n}{2} $ and $n \geq 2$.  Assume that  $p$ satisfies the following condition:
\begin{eqnarray*}
\max\Big\{\frac{n+1}{n-1}; \sigma; 2
 \Big\} <p <\infty.
 \end{eqnarray*}
  Then, there exists a constant $\varepsilon_0>0$ such that, for every given small data satisfying
  \[\|f\|_{H^\sigma}+ \|g\|_{ H^{\sigma-1}} \leq \varepsilon\,\,\,\mbox{for}\,\,\, \varepsilon\leq \varepsilon_0,\]
 there exists a uniquely determined global (in time) energy solution \[ \phi \in C\big([0,\infty),H^\sigma(\R^n)\big)\cap C^1\big([0,\infty), H^{\sigma-1}(\R^n)\big).\]  The energy solution satisfies the decay estimate
  \begin{eqnarray*}
  \|\phi(t,\cdot)\|_{H^\sigma} + \|\phi_t(t,\cdot)\|_{H^{\sigma-1}}  \lesssim e^{-\frac{n-1}{2}t}
  \big(\|f\|_{H^\sigma} + \|g\|_{H^{\sigma-1}} \big).
  \end{eqnarray*}
 \end{theorem}
For space dimensions $n=1,2$ we have a similar result as in Theorem \ref{onedimension}.
 \begin{theorem}\label{onedimension2} Consider the Cauchy problem  \eqref{general} for $m\in (\frac{n}{2},\infty)$ and data $(f,g) \in (H^\sigma(\R^n)\times H^{\sigma-1}(\R^n))$, with $\sigma=1$ for $n=1$ and  $1<\sigma <  2 $ for $n=2$. Let $p>\sigma +1$.   Then, there exists a constant $\varepsilon_0>0$ such that, for every given small data satisfying
  \[\|f\|_{H^\sigma}+ \|g\|_{ H^{\sigma-1}} \leq \varepsilon\,\,\,\mbox{for}\,\,\, \varepsilon\leq \varepsilon_0,\]
 there exists a uniquely determined global (in time) energy solution \[ \phi \in C\big([0,\infty),H^\sigma(\R^n)\big)\cap C^1\big([0,\infty), H^{\sigma-1}(\R^n)\big).\]  The energy solution satisfies the decay estimate
  \begin{eqnarray*}
  \|\phi(t,\cdot)\|_{H^\sigma} + \|\phi_t(t,\cdot)\|_{H^{\sigma-1}}  \lesssim e^{-\frac{n-1}{2}t}
  \big(\|f\|_{H^\sigma} + \|g\|_{H^{\sigma-1}} \big).
  \end{eqnarray*}
 \end{theorem}

\section{De Sitter model with balanced dissipation and mass: Case $m =\frac{n}{2}$.} \label{Sec5}
In this section we consider the Cauchy problem
\begin{equation}\label{semilinearwavetreatment}
\begin{cases}
u_{tt} - e^{-2t}\Delta u =e^{-\frac{n}{2} (p-1)t}|u|^p,
  \quad (t,x)\in (0,\infty)\times \R^n,
\\
u(0,x)= u_0(x), \quad u_t(0,x)=u_1(x),
  \quad x\in \R^n,
\end{cases}
\end{equation}
that is, $m =\frac{n}{2}$ in (\ref{general}).\\
According to Duhamel's principle, a solution of \eqref{semilinearwavetreatment} satisfies the non-linear integral equation
\[  u(t,x)=K_0(t,0, x) \ast_{(x)} u_0(x) + K_1(t,0, x) \ast_{(x)} u_1(x) + \int_0^t e^{-\frac{n}{2}(p-1)s}K_1(t,s,x) \ast_{(x)}|u(s,x)|^{p}\,ds, \]
where $K_j(t,0, x) \ast_{(x)} u_j(x)$, $j=0,1$, are the solutions to the corresponding linear Cauchy
problem
\begin{equation}\label{linearwave}
\begin{cases}
u_{tt} - e^{-2t}\Delta u =0,
  \quad (t,x)\in (0,\infty)\times \R^n,
\\
u(0,x)=\delta_{0j}u_0(x), \quad u_t(0,x)=\delta_{1j}u_1(x),
  \quad x\in \R^n,
\end{cases}
\end{equation}
with $\delta_{kj}=1$  for $k=j$, and zero otherwise. The term $K_1(t,s,x) \ast_{(x)}\psi(s,x)$ is the  solution  of the parameter-dependent Cauchy problem
\begin{equation}  \label{parameterdependentwave} \begin{cases}
u_{tt} - e^{-2t}\Delta u =0,
  \quad (t,x)\in (s,\infty)\times \R^n,
\\
u(s,x)=0, \quad u_t(s,x)=\psi(s,x),
  \quad x\in \R^n.
\end{cases}
\end{equation}
 So, Duhamel's principle explains that we have to take account of solutions
to a family of parameter-dependent Cauchy problems.
\subsection{Estimates of solutions to the corresponding linear model} \label{Sec5.1}
Let us consider the parameter-dependent Cauchy problem
\begin{equation}  \label{parameterdependentwave2} \begin{cases}
u_{tt} - e^{-2t}\Delta u =0,
  \quad (t,x)\in (s,\infty)\times \R^n,
\\
u(s,x)=\varphi(s,x), \quad u_t(s,x)=\psi(s,x),
  \quad x\in \R^n.
\end{cases}
\end{equation}
We perform the partial Fourier transformation with respect to the spatial variables to \eqref{parameterdependentwave2} and the change of variables
\[ v(\tau)=\widehat{u}(t,\xi), \quad \tau=|\xi|A(t)=|\xi|e^{-t}, \]
 leads to the Bessel equation of order zero
\[v_{\tau\tau} +\frac{1}{\tau} v_{\tau}+ v=0,
  \quad
v(\xii e^{-s})=\widehat\varphi(s,\xi),  \quad v_{\tau}(\xii e^{-s})=-\frac{\widehat\psi(s,\xi)}{|\xi|e^{-s}}.\]
The general solution of the Bessel equation of order zero is given for $\tau>0$ by
\[v(\tau)=c_1(\xi)J_0(\tau)+ c_2(\xi)Y_0(\tau),\]
where $J_{\nu}$ and $Y_{\nu}$
are Bessel functions of order $\nu$, of first and second kind, respectively. The  Wronskian  satisfies (\cite{BE}, vol.2, p.79)
\[W(J_0(\tau), Y_0(\tau))=\frac2{\pi\tau}.\]
Therefore, we obtain the following
representation:
\begin{eqnarray}\label{repre}
\begin{cases} \widehat{u}(t,\xi)=\frac{\pi}{2}|\xi|e^{-s}\Big(Y_0'(e^{-s}\xii)J_0(|\xi|e^{-t})-J_0'(e^{-s}\xii)Y_0(|\xi|e^{-t})\Big)
\widehat\varphi(s,\xi) \\
\quad +\frac{\pi}{2}\Big(Y_0(e^{-s}\xii)J_0(|\xi|e^{-t})-J_0(e^{-s}\xii)Y_0(|\xi|e^{-t})\Big)\widehat\psi(s,\xi).
\end{cases}
\end{eqnarray}
To describe the asymptotic behavior of $\widehat{u}$ we may use the following well-known properties of
the functions $J_0$ and $Y_0$ (see \cite{BE}, Vol.2):
\begin{itemize}
\item (B1): $J_0$ is an entire analytic function, whereas $Y_0$ has a logarithmic singularity at $\tau=0$;
\item (B2): $J_0'(\tau)=-J_1(\tau)$;
\item (B3): the behavior for small $\tau$ is given by
\[ |J_0(\tau)| \lesssim 1, \quad |J_0'(\tau)| \lesssim \tau, \quad |Y_0(\tau)| \lesssim |\log(\tau)|, \quad |Y_0'(\tau)| \lesssim \tau^{-1};\]
\item (B4): the behavior for large $\tau$ is given by
\[ |J_0^{(k)}(\tau)| \lesssim \tau^{-1/2}, \quad |Y_0^{(k)}(\tau)| \lesssim \tau^{-1/2}, \  k=0,1.\]
\end{itemize}
In order to have a pointwise estimate for $\widehat{u}$ we split again the extended phase space into three zones
\[ Z_1(s)
     := \{\xi: \ \xii e^{-s}\leq N \},\ \ Z_2(t,s)
     := \{N e^{s}\leq \xii \leq N e^{t} \},  \ \ Z_3(t)
     := \{\xii e^{-t}\geq N \}.\]
Let us denote by $t_{\xi}$ the separate line between $Z_2$ and $Z_3$, i.e., $\xii e^{-t_{\xi}}= N$.
The following result can be concluded by the results of \cite{ERnew}.
\begin{proposition}\label{propositionpurewave1}
Assume that $\varphi(0,\cdot) \in \dot{H}^{\gamma}(\R^n)$, $\gamma \geq 0$ and $\psi(0,\cdot) \equiv 0$. Then the following estimates hold for the solutions to (\ref{parameterdependentwave2}) for $t \in [0,\infty)$ and $\gamma \geq 0$:
\begin{eqnarray*}
&& \| K_0(t,0, x) \ast_{(x)} \varphi(0,x)\|_{\dot{H}^{\gamma}} \lesssim
  (1+t)e^{\frac{t}{2}} \|\varphi(0,\cdot)\|_{\dot{H}^{\gamma}},  \\
&&  \|\partial_t K_0(t,0, x) \ast_{(x)} \varphi(0,x)\|_{\dot{H}^{\gamma-1}} \lesssim
  \| \varphi(0,\cdot)\|_{\dot{H}^{\gamma}}.
\end{eqnarray*}
\end{proposition}
Moreover, we need the following result.
\begin{proposition}\label{propositionwave2parameter}
Assume that $\psi(s,\cdot) \in \dot{H}^\gamma(\R^n),\,\gamma \geq 0,$ for $s \in [0,\infty)$ and $\varphi(s,\cdot) \equiv 0$. Then the following estimates hold for the solutions to (\ref{parameterdependentwave2}) for $t \in [s,\infty)$:
\begin{eqnarray*}
&& \|K_1(t,s, x) \ast_{(x)} \psi(s,x)\|_{\dot{H}^\gamma} \lesssim
(1+t) \|\psi(s,\cdot)\|_{\dot{H}^\gamma},\\
&& \|\nabla K_1(t,s, x) \ast_{(x)} \psi(s,x)\|_{\dot{H}^\gamma} \lesssim
  (1+t) e^{\frac{s+t}{2}}  \|\psi(s,\cdot)\|_{\dot{H}^\gamma}, \\
&& \|\partial_t K_1(t,s, x) \ast_{(x)} \psi(s,x)\|_{\dot{H}^\gamma} \lesssim
  \|\psi(s,\cdot)\|_{\dot{H}^\gamma}.
\end{eqnarray*}
\end{proposition}
\begin{proof}
By using the property $(B3)$ in $Z_1$ we have
\[ \xii^{\gamma + j} |\hat{u}(t,\xi)| \lesssim
 (1+t)e^{js} \xii^{\gamma }\widehat\psi(s,\xi), \qquad j=0,1,\]
 and
 \[ \xii^{\gamma } |\hat{u}_t(t,\xi)| \lesssim
  \xii^{\gamma }\widehat\psi(s,\xi).\]
In $Z_3$ we have $\xii e^{-s}\geq \xii e^{-t}\geq N$ and by using the property $(B4)$ we can estimate
\[ \xii^{\gamma + j} |\hat{u}(t,\xi)| \lesssim
 e^{\frac{s-t}{2}+jt} \xii^{\gamma }\widehat\psi(s,\xi), \qquad j=0,1,\]
 and
 \[ \xii^{\gamma } |\hat{u}_t(t,\xi)| \lesssim
 e^{\frac{s-t}{2}} \xii^{\gamma }\widehat\psi(s,\xi).\]
In $Z_2$ we still have $\xii e^{-s}\geq N$ and by using the properties $(B3)$ and $(B4)$ we have that
\[ \xii^{\gamma + j} |\hat{u}(t,\xi)| \lesssim
 (1+t)e^{\frac{(s+t)j}{2}} \xii^{\gamma }\widehat\psi(s,\xi), \qquad j=0,1,\]
 and
 \[ \xii^{\gamma } |\hat{u}_t(t,\xi)| \lesssim
  \xii^{\gamma }\widehat\psi(s,\xi),\]
 thanks to $\xii e^{-t}\leq N$.
Summarizing all these estimates the proof of Proposition \ref{propositionwave2parameter} is concluded.
\end{proof}
\subsection{Global existence of small data solutions} \label{Sec5.2}
If we compare the estimates of Propositions \ref{propositionKleinGordon1}
and \ref{propositionKleinGordon2parameter} with those ones of Propositions \ref{propositionpurewave1}
and \ref{propositionwave2parameter}, then they differ only by the factor $1+t$ appearing in some of the estimates in
Propositions \ref{propositionpurewave1} and \ref{propositionwave2parameter}. So we can follow all the considerations of Section \ref{Sec4}. We study power non-linearities. The factor $1+t$ will not have any influence on the admissible set of powers $p$. Consequently, we do not expect any changes in the admissible set of powers $p$. For this reason
we mainly restrict ourselves to formulate the results only. We sketch modifications in the proof for one result.
Firstly we are interested in the global existence (in time) of energy solutions.
\begin{theorem}\label{main61} Consider for $n \geq 2$ the Cauchy problem  \eqref{general} with data $(f,g) \in (H^1(\R^n)\times L^2(\R^n))$. Let $p>\frac{n+1}{n-1}$ and $p\leq \frac{n}{n-2}$ for $n\geq 3$. Assume that the parameters $m$ and $n$ satisfy  $m=\frac{n}{2}$.
  Then, there exists a constant $\varepsilon_0>0$ such that, for every small data satisfying
  \[\|f\|_{H^1}+ \|g\|_{L^2}\leq \varepsilon\,\,\,\mbox{for}\,\,\, \varepsilon\leq \varepsilon_0,\]
  there exists a uniquely determined global (in time) energy solution \[ \phi \in C\big([0,\infty),H^1(\R^n)) \cap C^1([0,\infty),L^2(\R^n)\big).\]
  The energy solution satisfies the decay estimate
  \begin{eqnarray*}
  \|\phi(t,\cdot)\|_{L^2} + \|\nabla \phi(t,\cdot)\|_{L^2}+ \|\phi_t(t,\cdot)\|_{L^2}  \lesssim (1+t) e^{-\frac{n-1}{2}t}
  \big(\|f\|_{H^1} + \|g\|_{L^2} \big).
  \end{eqnarray*}
  \end{theorem}
  \begin{proof}
We only sketch the proof, in particular, we explain modifications to the proof of Theorem \ref{main3}. It is enough to prove the global existence of small data solutions to \eqref{semilinearwavetreatment}.
Motivated by the estimates of Proposition \ref{propositionpurewave1} and Proposition \ref{propositionwave2parameter} for $s=0$ we define for $t>0$ the scale of spaces of energy solutions
\begin{eqnarray*} && X(t) := \bigl\{ u\in C\big([0,t], H^1(\R^n) ) \cap C^1([0,t],L^2(\R^n)\big)\\ && \quad : \, \|u\|_{X(t)} :=\sup_{\tau\in[0,t]}\bigl\{(1+\tau)^{-1}e^{-\frac{\tau}{2}}\|u(\tau,\cdot)\|_{L^2} + (1+\tau)^{-1}e^{-\frac{\tau}{2}}\|\nabla u(\tau,\cdot)\|_{L^2} +\|u_\tau(\tau,\cdot)\|_{L^2} \bigr\}
<\infty \bigr\}. \end{eqnarray*}
In the following we use the same notations as in the proofs before.
Using Proposition \ref{propositionpurewave1} and Proposition \ref{propositionwave2parameter} for $s=0$ we have
\begin{equation}\label{eq:Kdatapurewave}
\| K_0(t,0, x) \ast_{(x)} u_0(x) + K_1(t,0, x) \ast_{(x)} u_1(x) \|_{X(t)} \lesssim\,\|u_0\|_{H^1} + \|u_1\|_{L^2}.
\end{equation}
 Using Minkowski's integral inequality gives with Proposition \ref{propositionwave2parameter}
 \[(1+t)^{-1}\|Gu(t,\cdot)\|_{L^2} \lesssim (1+t)^{-1}\int_0^t e^{-\frac{n}{2}(p-1)s}\|K_1(t,s,x) \ast_{(x)}|u(s,x)|^{p}\|_{L^2}\,ds\lesssim \int_0^t e^{-\frac{n}{2}(p-1)s}
 \| |u(s,x)|^{p}\|_{L^2}\,ds.\]
We may estimate
\[\| |u(s,\cdot)|^{p}\|_{L^2}=\|u(s,\cdot)\|_{L^{2p}}^p\lesssim
\| u(s,\cdot)\|_{L^2}^{p(1-\theta)} \| \nabla u(s,\cdot)\|_{L^2}^{p\theta} \lesssim (1+s)^p e^{\frac{s}{2}p}\|u\|_{X(s)}^p.\]
Hence,
\begin{eqnarray*}  && (1+t)^{-1}\|Gu(t,\cdot)\|_{L^2} \lesssim \|u\|_{X(t)}^p \int_0^t e^{-\frac{n}{2}(p-1)s+\frac{s}{2}p} (1+s)^p ds \\
&& \qquad =\|u\|_{X(t)}^p \int_0^t e^{-\frac{n-1}{2}(p-1)s+\frac{s}{2}} (1+s)^p ds
 \lesssim  e^{\frac{t}{2}}\|u\|_{X(t)}^p,\end{eqnarray*}
thanks to $p>1$. In the same way we get after applying Propositions \ref{propositionpurewave1} and \ref{propositionwave2parameter} to
\[\partial_t Gu(t,x)=\int_0^t e^{-\frac{n}{2}(p-1)s} \partial_t (K_1(t,s,x) \ast_{(x)}|u(s,x)|^{p})\,ds\]
the estimate
\[\| \partial_t   Gu(t,\cdot)\|_{L^2} \lesssim \|u\|_{X(t)}^p \int_0^t e^{-\frac{n}{2}(p-1)s+\frac{s}{2}p} (1+s)^p ds \lesssim \|u\|_{X(t)}^p \]
for all $p> \frac{n}{n-1}$. Finally, it remains to estimate
\[(1+t)^{-1} e^{-\frac{t}{2}}\nabla Gu(t,x)=(1+t)^{-1} e^{-\frac{t}{2}}\int_0^t e^{-\frac{n}{2}(p-1)s} \nabla (K_1(t,s,x) \ast_{(x)}|u(s,x)|^{p})\,ds.\]
Again the application of Propositions \ref{propositionpurewave1} and \ref{propositionwave2parameter} yields
\begin{eqnarray*}
&& \| (1+t)^{-1} e^{-\frac{t}{2}}\nabla Gu(t,\cdot)\|_{L^2} \lesssim (1+t)^{-1} e^{-\frac{t}{2}}\int_0^t e^{-\frac{n}{2}(p-1)s} \|\nabla (K_1(t,s,x) \ast_{(x)}|u(s,x)|^{p})\|_{L^2}\,ds \\
&& \qquad \lesssim  e^{-\frac{t}{2}}\int_0^t e^{-\frac{n}{2}(p-1)s +  \frac{s}{2} p + \frac{t+s}{2}} (1+s)^p \,ds \|u\|_{X(t)}^p
\\
&& \qquad \lesssim \int_0^t e^{-\frac{n}{2}(p-1)s +  \frac{s}{2} p + \frac{s}{2}} (1+s)^{p}\,ds \|u\|_{X(t)}^p \lesssim \|u\|_{X(t)}^p
\end{eqnarray*}
thanks to the assumption $p > \frac{n+1}{n-1}$.
Summarizing all the derived estimates it follows for all $t>0$ and $p>\frac{n+1}{n-1}$ the estimate
\[ \|Pu\|_{X(t)}
     \lesssim\,\|u_0\|_{H^1} + \|u_1\|_{L^2}+ \|u\|_{X(t)}^{p}.
     \]
This leads to $Pu\in X(t)$. Following the steps to show the Lipschitz property from the proof to Theorem \ref{main2} implies
     \begin{equation}
\label{eq:contractionwave}
\|Pu-Pv\|_{X(t)}
     \lesssim\|u-v\|_{X(t)} \bigl(\|u\|_{X(t)}^{p-1}+\|v\|_{X(t)}^{p-1}\bigr)
\end{equation}
for any $u,v\in X(t)$. Using these estimates for $Pu$ one can prove the existence of a
uniquely determined global (in time) energy solution $u$ by contraction argument for small data. Moreover, we get a local (in time) result for large data. The decay estimates of Theorem \ref{main61} follow by using the relation $\phi(t,x)=e^{-\frac{n}{2}t}u(t,x)$.
\end{proof}
Moreover, we can prove the following results.
\begin{theorem}\label{main62} Consider for $n \geq 2$ the Cauchy problem \eqref{general} with data $f \in H^\gamma(\R^n), \gamma \in (\frac{1}{2},1)$, and $g \in L^2(\R^n)$. Let $p>p_{n,\gamma}:=1+\frac{2\gamma}{n-1}$ and $p \leq \frac{n}{n-2\gamma}$.
 Assume that the parameters $m$ and $n$ satisfy $m=\frac{n}{2}$.
  Then, there exists a constant $\varepsilon_0>0$ such that, for every small data satisfying
  \[\|f\|_{H^\gamma}+ \|g\|_{L^2}\leq \varepsilon\,\,\,\mbox{for}\,\,\, \varepsilon\leq \varepsilon_0,\]
  there exists a uniquely determined global (in time) Sobolev solution \[ \phi \in C\big([0,\infty),H^\gamma(\R^n)\big).\] The solution satisfies the decay estimate  \begin{equation*}
  \| \phi(t,\cdot)\|_{H^\gamma} \lesssim
  (1+t)e^{-\frac{n-1}{2}t}\big(\|f\|_{H^\gamma}+ \|g\|_{L^2} \big).
 \end{equation*}
  \end{theorem}
\begin{theorem}\label{main63} Consider the Cauchy problem  \eqref{general} with data $(f,g) \in (H^\sigma(\R^n)\times H^{\sigma-1}(\R^n))$ for $n \geq 3$, where $\sigma \in \big(1,\frac{n}{2}\big)$.  Assume that the parameters $m$ and $n$ satisfy $m =\frac{n}{2}$. Finally, let $p$ satisfy the following condition:
\begin{eqnarray*}  \max\Big\{\frac{n+1}{n-1};\lceil \sigma
 \rceil\Big\} <p \leq 1+\frac{2}{n-2\sigma}.
 \end{eqnarray*}
  Then, there exists a constant $\varepsilon_0>0$ such that, for every small data satisfying
  \[\|f\|_{H^\sigma}+ \|g\|_{H^{\sigma-1}}\leq \varepsilon\,\,\,\mbox{for}\,\,\, \varepsilon\leq \varepsilon_0,\]
  there exists a uniquely determined global (in time) energy solution \[ \phi \in C\big([0,\infty),H^\sigma(\R^n)) \cap C^1([0,\infty),H^{\sigma-1}(\R^n)\big).\]
  The energy solution satisfies the decay estimate
  \begin{eqnarray*}
  \|\phi(t,\cdot)\|_{H^\sigma} + \|\phi_t(t,\cdot)\|_{H^{\sigma-1}}  \lesssim (1+t)e^{-\frac{n-1}{2}t}
  \big(\|f\|_{H^\sigma} + \|g\|_{H^{\sigma-1}} \big).
  \end{eqnarray*}
  \end{theorem}
  The statement of the previous theorem implies for $n \leq 2\sigma$ the following result.
\begin{corollary} \label{Corlargeregularitywave}
Consider the Cauchy problem  \eqref{general} with data $(f,g) \in (H^\sigma(\R^n)\times H^{\sigma-1}(\R^n))$ for $n \geq 3$ and
$n \leq 2\sigma$. Assume that the parameters $m$ and $n$ satisfy $ m=\frac{n}{2}$.
Finally, let $p$ satisfy the following condition:
\begin{eqnarray*}  \max\Big\{\frac{n+1}{n-1};\lceil \sigma
 \rceil\Big\} <p<\infty.
 \end{eqnarray*}
  Then, there exists a constant $\varepsilon_0>0$ such that, for every small data satisfying
  \[\|f\|_{H^\sigma}+ \|g\|_{H^{\sigma-1}}\leq \varepsilon\,\,\,\mbox{for}\,\,\, \varepsilon\leq \varepsilon_0,\]
  there exists a uniquely determined global (in time) energy solution \[ \phi \in C\big([0,\infty),H^\sigma(\R^n)) \cap C^1([0,\infty),H^{\sigma-1}(\R^n)\big).\]
  The energy solution satisfies the decay estimate
  \begin{eqnarray*}
  \|\phi(t,\cdot)\|_{H^\sigma} + \|\phi_t(t,\cdot)\|_{H^{\sigma-1}}  \lesssim (1+t)e^{-\frac{n-1}{2}t}
  \big(\|f\|_{H^\sigma} + \|g\|_{H^{\sigma-1}} \big).
  \end{eqnarray*}
\end{corollary}
Similarly to Theorem \ref{improve1} we may improve the lower bound for $p$ in Corollary \ref{Corlargeregularitywave}.
   \begin{theorem}\label{improve71} Consider the Cauchy problem  \eqref{general} with $m=\frac{n}{2}$ and data $(f,g) \in (H^\sigma(\R^n)\times H^{\sigma-1}(\R^n))$ for
 $\sigma > \frac{n}{2} $ and $n \geq 2$.  Assume that  $p$ satisfies the following condition:
\begin{eqnarray*}
\max\Big\{\frac{n+1}{n-1}; \sigma; 2
 \Big\} <p <\infty.
 \end{eqnarray*}
  Then, there exists a constant $\varepsilon_0>0$ such that, for every given small data satisfying
  \[\|f\|_{H^\sigma}+ \|g\|_{ H^{\sigma-1}} \leq \varepsilon\,\,\,\mbox{for}\,\,\, \varepsilon\leq \varepsilon_0,\]
 there exists a uniquely determined global (in time) energy solution \[ \phi \in C\big([0,\infty),H^\sigma(\R^n)\big)\cap C^1\big([0,\infty), H^{\sigma-1}(\R^n)\big).\]  The energy solution satisfies the decay estimate
  \begin{eqnarray*}
  \|\phi(t,\cdot)\|_{H^\sigma} + \|\phi_t(t,\cdot)\|_{H^{\sigma-1}}  \lesssim (1+t) e^{-\frac{n-1}{2}t}
  \big(\|f\|_{H^\sigma} + \|g\|_{H^{\sigma-1}} \big).
  \end{eqnarray*}
 \end{theorem}
For space dimensions $n=1,2$ we have a similar result as in Theorem \ref{onedimension}.
 \begin{theorem}\label{onedimension21} Consider the Cauchy problem  \eqref{general} for $m=\frac{n}{2}$ and data $(f,g) \in (H^\sigma(\R^n)\times H^{\sigma-1}(\R^n))$, with $\sigma=1$ for $n=1$ and  $1<\sigma <  2 $ for $n=2$. Let $p>\sigma +1$.   Then, there exists a constant $\varepsilon_0>0$ such that, for every given small data satisfying
  \[\|f\|_{H^\sigma}+ \|g\|_{ H^{\sigma-1}} \leq \varepsilon\,\,\,\mbox{for}\,\,\, \varepsilon\leq \varepsilon_0,\]
 there exists a uniquely determined global (in time) energy solution \[ \phi \in C\big([0,\infty),H^\sigma(\R^n)\big)\cap C^1\big([0,\infty), H^{\sigma-1}(\R^n)\big).\]  The energy solution satisfies the decay estimate
  \begin{eqnarray*}
  \|\phi(t,\cdot)\|_{H^\sigma} + \|\phi_t(t,\cdot)\|_{H^{\sigma-1}}  \lesssim (1+t) e^{-\frac{n-1}{2}t}
  \big(\|f\|_{H^\sigma} + \|g\|_{H^{\sigma-1}} \big).
  \end{eqnarray*}
 \end{theorem}

{\bf Acknowledgements} $\quad$
The discussions on this paper began during the time the first author spent his sabbatical year (July 2014 - July 2015) at the Institute of Applied Analysis at TU Bergakademie Freiberg. The stay of the first author was  supported by Funda\c{c}\~{a}o de Amparo \`{a} Pesquisa do Estado de S\~{a}o Paulo (FAPESP), grant 2013/20297-8. This paper was completed within the DFG project RE 961/21-1. The authors thank Karen Yagdjian (Edinburg) for fruitful discussions on the content of this paper.

\section{Appendix} \label{SecAppendix}
In the Appendix we list some results of Harmonic Analysis which are important tools for proving results on the global existence of small data solutions for semi-linear de Sitter models with power non-linearities. In particular, these are tools which allow to estimate power non-linearities in homogeneous Sobolev spaces (see \cite{Palmierithesis}). First of all we introduce the Bessel and Riesz
potential spaces.

\subsection{Bessel and Riesz potential spaces} \label{SecBesselRiesz}
Let $s\in \mathbb{R}$ and $1<p<\infty$. Then
\begin{align*}
H^s_p(\mathbb{R}^n)&=\{u\in\mathcal{S}'(\mathbb{R}^n): \|\langle D \rangle^s u\|_{L^p(\mathbb{R}^n)}=\| u\|_{H^s_p(\mathbb{R}^n)}<\infty \},\\
\dot{H}^s_p(\mathbb{R}^n)&=\{u\in\mathcal{Z}'(\mathbb{R}^n): \||D|^s u\|_{L^p(\mathbb{R}^n)}=\| u\|_{\dot{H}^s_p(\mathbb{R}^n)}<\infty \}
\end{align*}
are called Bessel and Riesz potential spaces, respectively. If $p=2$, then we use the notations $H^s(\mathbb{R}^n)$ and
 $\dot{H}^s(\mathbb{R}^n)$, respectively. In the definition of the Riesz potential spaces we use the space of distributions $\mathcal{Z}'(\mathbb{R}^n)$.
This space of distributions can be identified with the factor space $\mathcal{S}'/\mathcal{P}$, where $\mathcal{P}$ denotes the set of all polynomials.

\subsection{Fractional Gagliardo-Nirenberg inequality} \label{SecGagliardoNirenbergfractional}

The first inequality that we present is a generalization of the classical Gagliardo-Nirenberg inequality to the case of Sobolev spaces of fractional order. Therefore, we will refer to the upcoming result as {\it fractional Gagliardo-Nirenberg inequality}.

\begin{proposition}\label{fractionalGagliardoNirenberg}
Let $1<p,p_0,p_1<\infty$, $\sigma >0$ and $s\in [0,\sigma)$. Then it holds the following fractional Gagliardo-Nirenberg inequality for all $u\in L^{p_0}(\mathbb{R}^n)\cap \dot{H}^\sigma_{p_1}(\mathbb{R}^n)$:
\begin{align} \label{PropositionFGNI}
\|u\|_{\dot{H}^{s}_p}\lesssim \|u\|_{L^{p_0}}^{1-\theta}\|u\|_{\dot{H}^{\sigma}_{p_1}}^\theta,
\end{align} where $\theta=\theta_{s,\sigma}(p,p_0,p_1)=\frac{\frac{1}{p_0}-\frac{1}{p}+\frac{s}{n}}{\frac{1}{p_0}-\frac{1}{p_1}+\frac{\sigma}{n}}$ and $\frac{s}{\sigma}\leq \theta\leq 1$ .
\end{proposition}
For the proof one can see \cite{Ozawa}.

\begin{corollary} Let $1<p,m<\infty$, $\sigma >0$ and $s\in[0,\sigma)$. Then we have the following inequality for all $u\in H^\sigma(\mathbb{R}^n)$:
\begin{equation}
\|\,|D|^{s} u\|_{L^p}\lesssim \| u\|_{L^m}^{1-\theta}\|\,|D|^\sigma u\|_{L^m}^{\theta},
\end{equation} where $\theta=\theta_{s,\sigma}(p,m)=\frac{n}{\sigma}\big(\frac{1}{m}-\frac{1}{p}+\frac{s}{n}\big)$ and $ \frac{s}{\sigma}\leq \theta_{s,\sigma}(p,m)\leq 1$.
\end{corollary}

\subsection{Fractional Leibniz rule} \label{SecLeibnizfractional}


\begin{proposition} \label{fractionalLeibniz}
Let us assume $s>0$ and $1\leq r \leq \infty, 1<p_1,p_2,q_1,q_2 < \infty$ satisfying the relation \[ \frac{1}{r}=\frac{1}{p_1}+\frac{1}{p_2}=\frac{1}{q_1}+\frac{1}{q_2}.\]
Then the following fractional Leibniz rules hold:
\begin{align}\label{PropFLR} 
\|\,|D|^s(u \,v)\|_{L^r}\lesssim \|\,|D|^s u\|_{L^{p_1}}\|v\|_{L^{p_2}}+\|u\|_{L^{q_1}}\|\,|D|^s v\|_{L^{q_2}}
\end{align}  for any $u\in \dot{H}^s_{p_1}(\mathbb{R}^n)\cap L^{q_1}(\mathbb{R}^n)$ and $v\in \dot{H}^s_{q_2}(\mathbb{R}^n)\cap L^{p_2}(\mathbb{R}^n)$,
\begin{align}\label{PropFLR100} 
\|\langle D \rangle^s(u \,v)\|_{L^r}\lesssim \|\langle D \rangle^s u\|_{L^{p_1}}\|v\|_{L^{p_2}}+\|u\|_{L^{q_1}}\|\langle D\rangle^s v\|_{L^{q_2}}
\end{align}  for any $u\in H^s_{p_1}(\mathbb{R}^n)\cap L^{q_1}(\mathbb{R}^n)$ and $v\in H^s_{q_2}(\mathbb{R}^n)\cap L^{p_2}(\mathbb{R}^n)$.
\end{proposition}
These results can be found in \cite{Grafakos}.

\subsection{Fractional chain rule} \label{Secchainrulefractional}

\begin{proposition}\label{fractionalchainruleChrist}
Let us choose $s\in (0,1)$, $1<r,r_1,r_2<\infty$ and a $C^1$ function $F$ 
 satisfying for any $\tau \in [0,1]$ and $u,v\in \mathbb{R}$ the inequality
\begin{align}\label{F' assumption}
|F'(\tau u+(1-\tau) v)|\leq \mu(\tau) (G(u)+G(v)),
\end{align}  for some continuous and non-negative function $G$
and some non-negative function $\mu\in L^1([0,1])$.\\
Under these assumptions the following estimate is true:
\begin{align}\label{FCR Christ version}
\|F(u)\|_{\dot{H}^s_r}\lesssim \| G(u)\|_{L^{r_1}}\|u\|_{\dot{H}^s_{r_2}}
\end{align} for any $u\in  \dot{H}^s_{r_2}(\mathbb{R}^n)$ such that $G(u)\in L^{r_1}(\mathbb{R}^n)$, provided that
\begin{align*}
\frac{1}{r}=\frac{1}{r_1}+\frac{1}{r_2}.
\end{align*}
\end{proposition}
For the proof of this result one can see \cite{Christ} or the proof in a slightly modified version in \cite{Palmierithesis}.\\
In particular we can apply Proposition \ref{fractionalchainruleChrist} for  $F(u)=|u|^p$ or $F(u)=\pm u|u|^{p-1}$. After choosing $G(u)=|F'(u)|$ and $\mu$ as a positive constant the next result
follows immediately.
\begin{corollary}\label{CorfractionalchainruleChrist} Let $F(u)=|u|^p$ or $F(u)=\pm u|u|^{p-1}$ for $p>1$, $s\in (0,1)$ and $r,r_1,r_2\in (1,\infty)$.
Then,
\begin{align*}
\|F( u)\|_{\dot{H}^s_r}\lesssim \|  u\|^{p-1}_{L^{r_1}}\|u\|_{\dot{H}^s_{r_2}}
\end{align*} for any $u\in L^{r_1}(\mathbb{R}^n)\cap H^s_{r_2}(\mathbb{R}^n)$, provided that
\begin{align*}
\frac{1}{r}=\frac{p-1}{r_1}+\frac{1}{r_2}.
\end{align*}
\end{corollary}
The following result shows that there is no necessity to assume $s \in (0,1)$ in the last corollary. In the formulation we use for $s \in \mathbb{R}$ the symbol $\lceil s \rceil$ which denotes the smallest integer greater than or equal to $s$.
\begin{proposition}\label{Propfractionalchainrulegeneral} Let us choose $s>0$, $p>\lceil s \rceil$
 and $1<r,r_1,r_2<\infty$ satisfying \[ \frac{1}{r}=\frac{p-1}{r_1}+\frac{1}{r_2}.\]
Let us denote by $F(u)$ one of the functions $|u|^p, \pm |u|^{p-1}u$.\\
 Then it holds the following fractional chain rule:
\begin{align}\label{fractionalchainrulegeneral}
\|\,|D|^{s} F(u)\|_{L^r}\lesssim \|u\|_{L^{r_1}}^{p-1}\|\,|D|^{s} u\|_{L^{r_2}}
\end{align}
 for any $u\in  L^{r_1}(\mathbb{R}^n)\cap \dot{H}^{s}_{r_2}(\mathbb{R}^n)$.
 \end{proposition}
The proof can be found in \cite{Palmierithesis}. It uses an induction argument and among other things the following auxiliary result whose proof was provided by Winfried Sickel (University of Jena) within a private communication.
\begin{lemma}\label{Sickelauxiliary lemma}
Let $\omega \geq 0$ and $q \in (1,\infty)$. Then we have the following equivalence of norms:
\[ \||D|^\omega \nabla u \|_{L^q} \approx \||D|^{\omega+1} u \|_{L^q} \]
under the assumption, that $u$ is given so that both norms exist.
\end{lemma}

\subsection{Fractional powers} \label{Secfractionalpowers}
We apply a result from
\cite{RunSic} for fractional powers.
\begin{proposition} \label{PropSickelfractional} Let $p>1$, $1< r <\infty$ and $u \in H^{s}_r$, where $s \in \big(\frac{n}{r},p\big)$.
Let us denote by $F(u)$ one of the functions $|u|^p,\, \pm |u|^{p-1}u$ with $p>1$.\\
Then the following estimate holds$:$
$$\Vert F(u)\Vert_{H^{s}_r}\le C \|u\|_{H^{s}_r}\|u\|_{L^\infty}^{p-1}.$$
In particular, if $s\in \mathbb{N}$, one may weaken the condition on $p$ to $p> s-\frac1{r}$.
\end{proposition}
\noindent We shall use the following corollary from Proposition \ref{PropSickelfractional}.
\begin{corollary} \label{Corfractionalhomogeneous}
Under the assumptions of Proposition \ref{PropSickelfractional} it holds
$$\| F(u)\|_{\dot{H}^{s}_r}\le C \| u\|_{\dot{H}^{s}_r}\|u\|_{L^\infty}^{p-1}.$$
\end{corollary}
\begin{proof}
Let us prove it for $F(u)=|u|^p$. We write the estimate from Proposition \ref{PropSickelfractional} in the form $$   \| |u|^p\|_{\dot{H}^{s}_r}+ \| |u|^p\|_{L^{r}}\le C \big(\|
u\|_{\dot{H}^{s}_r}+ \|u\|_{L^{r}}\big)\|u\|_{L^\infty}^{p-1}.  $$ Using instead of $u$ the dilation $u_\lambda(\cdot):=u(\lambda \cdot)$ in the
last inequality we obtain the desired inequality after taking into consideration
\[  \|u_\lambda\|_{\dot{H}^{s}_r}=\lambda^{s-\frac{n}{r}}\|u\|_{\dot{H}^{s}_r}\,\,\,\mbox{and}\,\,\,
\|u_\lambda\|_{L^{r}}=\lambda^{-\frac{n}{r}} \|u\|_{L^{r}}      \] and letting $\lambda$ tend to infinity.
\end{proof}

\subsection{Fractional homogeneous Sobolev embeddings} \label{SecfractionalSobolevembedding}

 Sometimes one can apply in proofs for global existence results for semi-linear Cauchy problems, instead of the fractional Gagliardo-Nirenberg inequality  the embedding of a homogeneous fractional Sobolev space with suitable order $\dot{H}^\kappa$ in $L^q$, that is, the following result.
\begin{proposition} \label{PropfractionalSobolevembedding}
Let $q\geq 2$ and $\kappa=n\big(\frac{1}{2}-\frac{1}{q}\big)$. Then the following fractional Sobolev embedding is valid:
\begin{align*}
\dot{H}^\kappa(\mathbb{R}^n)\hookrightarrow L^q(\mathbb{R}^n).
\end{align*}
Therefore, there exists a constant $C=C(n,q)>0$ such that
\begin{align*}
\|u\|_{L^q}\leq  C\|u\|_{\dot{H}^\kappa}
\end{align*} for any $u\in \dot{H}^\kappa(\mathbb{R}^n)$.
\end{proposition}

\end{document}